%% file: main.tex
\documentclass[preprint,10pt,3p]{elsarticle}

\usepackage{amsmath,amssymb,amsthm}
\usepackage{booktabs}
\usepackage{graphicx}
\usepackage{multirow}
\usepackage[table]{xcolor}
\usepackage{algorithm2e}
\PassOptionsToPackage{hyphens}{url}
\usepackage[colorlinks=true,allcolors=blue]{hyperref}
\usepackage{url}
\usepackage{etoolbox}
\usepackage{placeins}
\AtBeginEnvironment{table}{\setlength{\belowcaptionskip}{3pt}}

\graphicspath{{figs/}}

\newtheorem{proposition}{Proposition}
\newtheorem{remark}{Remark}

\newcommand{\bx}{\mathbf{x}}
\newcommand{\ba}{\mathbf{a}}
\newcommand{\bF}{\mathbf{F}}

\newcommand{\refdom}{\widehat{\Omega}}
\newcommand{\uh}{u_h}
\newcommand{\ch}{c_h}
\newcommand{\uhyb}{u_{\mathrm{hyb}}}
\newcommand{\Kv}{K_v}
\newcommand{\Nv}{N_v}
\newcommand{\rhov}{\rho_v}
\newcommand{\xv}{\bx_v}
\newcommand{\dloc}{d_{\mathrm{loc}}}
\newcommand{\wdata}{w_{\mathrm{data}}}
\newcommand{\wbc}{w_{\mathrm{bc}}}
\newcommand{\eLtwo}{e_{L^2}}
\newcommand{\eHone}{e_{H^1}}
\newcommand{\Ltwo}{L^2}
\newcommand{\Hone}{H^1}
\newcommand{\sci}[2]{#1\times10^{#2}}
\newcommand{\codeurl}{https://github.com/Sima-Naraghi/iga-kan}

\journal{}

\begin{document}

\begin{frontmatter}

\title{IGA--KAN: Isogeometric Analysis with Physics-Informed Closed-Form Kolmogorov--Arnold Networks for Forward and Inverse PDEs}

\author[am]{Sima Naraghi}
\ead{s\_naraghi@sbu.ac.ir}

\author[cds,cog]{Kourosh Parand}
\ead{k\_parand@sbu.ac.ir}
\cortext[cor]{Corresponding author.}

\author[cse]{Amirhossein Sadr}
\ead{a.sadr@mail.sbu.ac.ir}

\author[cse]{Dara Rahmati}
\ead{d\_rahmati@sbu.ac.ir}

\address[am]{Department of Applied Mathematics, Faculty of Mathematical Sciences,
Shahid Beheshti University, Tehran, Iran}
\address[cds]{Department of Computer and Data Sciences, Faculty of Mathematical Sciences,
Shahid Beheshti University, Tehran, Iran}
\address[cog]{Department of Cognitive Modeling, Institute for Cognitive and Brain Sciences,
Shahid Beheshti University, Tehran, Iran}
\address[cse]{Department of Computer Science and Engineering, Shahid Beheshti University,
Tehran, Iran}

\begin{abstract}
Isogeometric analysis (IGA) solves partial differential equations accurately on exact NURBS geometry, whereas
neural solvers are mesh-free but often orders of magnitude less accurate and typically trained by non-convex
optimization without error control. We propose IGA--KAN, which uses local Kolmogorov--Arnold networks, fitted in
closed form, to improve the IGA solution instead of replacing it. An IGA Galerkin solve produces $\uh$; on
every knot-vertex patch a Kolmogorov--Arnold ridge model is fitted to the strong form of the equation, the
exact boundary data and $\uh$, and the models are blended by IGA hat functions. With fixed inner
functions the fit is one batched linear least-squares problem, without optimizer, learning rate or initialization.
An \emph{a posteriori} safeguard, motivated by a maximum-principle bound, decides where local models are used,
keeping the IGA solution elsewhere. On eight benchmarks with exact solutions, five from the literature and one
also posed on a domain fitted to a brain slice from MRI, the method reduces the error of IGA, at an unchanged
number of Galerkin unknowns, by factors of $4.2$ to $90$ in $L^2$ and $4.1$ to $220$ in $H^1$ on the reference
meshes, and its $L^2$ error is $6$ to $6\times10^{4}$ times smaller than that of the best Kolmogorov--Arnold
network trained from scratch on the same equations with a fixed budget. In an inverse problem it recovers an
unknown constant source from one noise-free observation $167$ times more accurately than IGA. The gain is
attributed to the superconvergence of local averages of the Galerkin solution.
\end{abstract}


\end{frontmatter}

\input{sections/introduction}
\input{sections/related_work}
\input{sections/method}
\input{sections/mechanism}

\input{sections/numerics}

\FloatBarrier
\input{sections/conclusions}

\section*{Code and data availability}
\begin{sloppypar}
The code and the result files of this study are openly available at \url{\codeurl}. The repository contains the
Python library \texttt{igakan} (the isogeometric solvers, the local Kolmogorov--Arnold fit with its safeguard, and
the stand-alone network baselines); one executed Jupyter notebook per benchmark, which solves it with IGA and
IGA--KAN and evaluates the network baselines; the scripts that produce every table and figure of this paper,
including those that train the networks; and the result files (CSV, JSON and NPZ, with the logs and, for the flower, the
Fisher--Kolmogorov equation, the heat equation and the coupled system, the trained network weights) from which every number reported here is read. The brain slice of Section~\ref{sec:fk} is taken
from the subject labeled \texttt{abby} of the public MRI2FEM data set \cite{mri2femdata} (version 2.0, CC BY 4.0),
which is not redistributed; the repository gives the instructions to obtain it.
\end{sloppypar}

\section*{Declaration of generative AI and AI-assisted technologies in the manuscript preparation process}
During the preparation of this work the authors used Claude (Anthropic) in order to edit the text and to check its
grammar and language. After using this tool, the authors reviewed and edited the content as needed and take full
responsibility for the content of the published article.

\bibliographystyle{elsarticle-num}
\setlength{\bibsep}{1pt plus 0.3ex}
\bibliography{refs}

\end{document}

%% file: sections/introduction.tex
\section{Introduction}\label{sec:intro}

Partial differential equations posed on curved geometries, whether described by CAD or segmented from medical
images, are at the core of computational mechanics, and for them classical Galerkin discretizations offer what an
engineering or clinical computation needs: convergence under refinement, error estimates and a cost that grows
predictably with the number of unknowns. Isogeometric analysis (IGA) adds the exact geometry: it uses the NURBS
description of a CAD geometry as the trial space of a Galerkin method, so that conic sections are represented
exactly and the discrete functions are $C^{p-1}$ across knot spans \cite{hughes2005isogeometric,cottrell2009isogeometric}.
Its accuracy, however, is paid for in unknowns: at a given degree, a smaller error requires a finer mesh and a larger
linear system. Machine learning has been proposed as a way to relax this trade-off.

Physics-informed neural networks (PINNs) train a network on the residual of the equation at collocation points
\cite{lagaris1998ann,raissi2019physics}. They are mesh-free and differentiable, but their training is a non-convex
optimization without error control, and its failure modes are well documented: spectral bias slows the learning of
high frequencies, and the loss landscape becomes ill-conditioned as convection or reaction coefficients grow
\cite{wang2021eigenvector,wang2022ntk,krishnapriyan2021failure}. A systematic comparison with the finite element
method on linear and nonlinear model problems found that PINNs did not outperform it in either accuracy or
solution time \cite{grossmann2024pinnfem}, and a systematic review of machine learning for fluid-related equations
found that most articles claiming to outperform a standard numerical method had compared against a weak baseline
\cite{mcgreivy2024weak}. For problems that classical solvers already handle well, the relevant question is
therefore not whether a network can replace the solver, but whether it can add accuracy to the solution the solver
produces.

Kolmogorov--Arnold networks (KANs) are a natural candidate for such a role, because their building block is close to
that of isogeometric analysis: learnable univariate splines replace fixed activations \cite{liu2024kan}, and other
univariate families, rational functions in rKAN \cite{aghaei2026rkan} and radial basis functions in Lynx-Net
\cite{mirzabeigi2026lynxnet}, have been proposed as well. So far, however, KANs have been used in partial
differential equations in the same role as other networks: inside a physics-informed loss. KINN
\cite{wang2025kinn}, across strong, energy and boundary-integral formulations, finds a KAN more accurate than a
multilayer perceptron on singularities and discontinuous coefficients and far less affected by spectral bias,
provided its spline grid is fine enough, but no more accurate than the perceptron on complex geometries, and
PG-KINN \cite{sadr2026pgkinn} moves it to a Petrov--Galerkin weak form. On the isogeometric side, networks of other
types produce the \emph{control net} of an isogeometric space \cite{li2023igapignn,moller2026iganets,gasick2023isogeometric},
which makes the geometry exact and the essential conditions strict, but bounds the accuracy by that of the
discretization being imitated. In both lines the network is the solver: it inherits the non-convex training, and no
classical solution is left to fall back on where its prediction is poor. Networks that do correct a classical
solution, by super-resolution, solver-in-the-loop training or a fine-scale correction of a coarse finite element
solution, are trained on reference solutions
\cite{fukami2019superresolution,um2020solver,kapustsin2023hybrid,margenberg2024dnnmg}, and the closest scheme that
needs no such data, FEM-LSSVR \cite{babaei2025superresolution}, post-processes a low-order finite element solution
element by element (Section~\ref{sec:related}). What is missing is a correction that needs no reference data, works
on the exact geometry and is used only where computable tests support it, with the isogeometric solution kept
everywhere else.

Such a correction is possible because a Galerkin solution carries more accuracy than its pointwise values show. Its
error oscillates on the scale of the elements, whereas, for degree $p\ge2$, its local averages are far more
accurate: the error superconverges in negative norms, which classical post-processors exploit by convolution
\cite{bramble1977higher}, and its derivatives superconverge at particular points, which superconvergent patch
recovery exploits by local fitting \cite{zienkiewicz1992spr} (Section~\ref{sec:related-superconv}). A smooth local
model fitted to $\uh$ keeps the accurate averages and discards the oscillation; if it is also required to satisfy the
equation and the exact boundary data, it restores the fine structure that the averages lack. A small KAN can serve as
such a model, provided that it is used according to the following principle:

\begin{quote}
\emph{Solve the problem classically; use the network only locally, only where its prediction passes a computable
check, and with its inner functions fixed, so that fitting it is a linear problem.}
\end{quote}

The last condition is a design choice of this paper, not a property of KANs: the KAN-based solvers above train all the
univariate functions of the network together by non-convex optimization, whereas fixing the inner functions, and
freezing or linearizing any nonlinear term at the classical solution (Section~\ref{sec:method:operators}), turns the
fit of the outer functions into a linear least-squares problem with a closed-form solution.

The method follows this principle in three steps. A standard isogeometric Galerkin solve produces $\uh$ on the exact
NURBS geometry. On every knot-vertex patch, one small Kolmogorov--Arnold model of \emph{ridge} type (learnable
univariate splines composed with fixed linear forms) is fitted to the strong form of the equation, to the exact
Dirichlet data on the NURBS boundary crossing the patch and to $\uh$ itself. The local models are then blended by the
degree-one hat functions of the isogeometric mesh. Because neighboring local models overlap, they can be
cross-examined: an \emph{a posteriori} safeguard decides, from computable quantities alone, where a local model is
used. The isogeometric solution is never discarded: it enters every local fit as data and serves as the exact
fallback of the safeguard. The method therefore retains the exact geometry and the Galerkin solution, and adds
accuracy only where computable tests support a local model. We call the result IGA--KAN or, for short, the hybrid: not a
neural solver, but an isogeometric solution with a \emph{safeguarded local post-process}.

We use the Kolmogorov--Arnold form for its structure, not for its expressivity. With the single-span outer splines
used throughout, the local model spans exactly the polynomials of its degree, and on the first benchmark of
Section~\ref{sec:examples} a monomial basis of the same space reproduces the results (Remark~\ref{rem:space}). In
this paper the form is therefore a parametrization, and the accuracy comes from what the model is fitted to: the
equation, the exact boundary data and $\uh$. We keep the form for two reasons. Every derivative of a ridge function
$\Phi(\ba\cdot\bx)$ is a multiple of a derivative of the univariate spline $\Phi$, so the equation rows of every
operator in this paper, in two and in three dimensions, are assembled from univariate splines and their derivatives
alone. And it places the local model in the Kolmogorov--Arnold family: its outer functions are B-splines, like the
trial functions of IGA and the edge functions of a KAN, and can be given interior knots to leave the polynomial
space while the fit stays linear, an option this paper does not use.

Across the eight benchmarks of Section~\ref{sec:examples} the hybrid reduces the error of its input, at an unchanged
number of unknowns of the Galerkin solve, by factors of $4.2$ to $90$ in the relative $\Ltwo$ norm and $4.1$ to
$220$ in the $\Hone$ norm on the reference meshes; the largest factors are those of the disk, whose quadratic
solution lies in the local space, and on the six benchmarks with non-polynomial solutions, all but the disk and the
thick ring, the factors are $4.2$ to $27$ and $4.1$ to $148$. Its $\Ltwo$ error is also $6$ to $6\times10^{4}$ times
smaller than that of the best Kolmogorov--Arnold network we trained from scratch on the same equations with a fixed
training protocol; these factors describe that protocol, not the best accuracy such networks can reach
(Section~\ref{sec:examples}). With the settings
used throughout, the hybrid did not return a larger error than its input, beyond round-off, on any mesh of these
benchmarks.

Improving a Galerkin solution by local fitting is a classical idea (Section~\ref{sec:related-superconv}). What this
paper adds is a local fit that also uses the equation and the exact geometry, takes the form of a Kolmogorov--Arnold
model computed in closed form and is used only where computable tests support it, together with an analysis of its
mechanism and numerical evidence on eight benchmarks. Specifically, the paper makes the following four contributions.

\begin{enumerate}
\item \textbf{A closed-form local Kolmogorov--Arnold post-process for isogeometric solutions.} The local model
  \eqref{eq:kan} has fixed inner functions and Bernstein outer functions, so its fit \eqref{eq:rows} to the
  equation, the boundary data and $\uh$ is one batched linear least-squares solve over all patches. The same
  construction covers advection--diffusion--reaction operators, the space--time heat operator, nonlinear terms
  frozen at or linearized about the classical solution, and triangularly coupled systems.
\item \textbf{An \emph{a posteriori} safeguard with an exact fallback.} Four computable tests (neighbor
  consensus, the size of the correction, a residual indicator and a $p$-consistency test) decide patch by patch
  whether a local model is used; where it is not, the isogeometric solution is kept. For elliptic operators the
  constant of the residual indicator is taken from a maximum-principle bound on the error of a local model
  (Proposition~\ref{prop:local}).
\item \textbf{An analysis of the mechanism.} The post-process is a local average whose kernel is chosen by the
  equation; Propositions~\ref{prop:blend} and~\ref{prop:local} bound the error of the blended solution by the
  residuals of the local models and by their errors on the patch boundaries, for which the residual and consensus
  tests are computable surrogates, and the analysis identifies the regimes in which the gain vanishes.
\item \textbf{Numerical evidence on eight benchmarks with exact solutions}, five of them taken from published
  work, on curved and complex NURBS geometries in two and three dimensions and on a ring-shaped domain fitted to a
  brain slice segmented from magnetic resonance images, with a comparison against isogeometric analysis and against
  Kolmogorov--Arnold networks (KAN, KINN) and the published results of PI-GGN, KINN and PG-KINN.
\end{enumerate}

The paper is organized as follows. Section~\ref{sec:related} reviews the literature, Section~\ref{sec:method} states
the method, Section~\ref{sec:mechanism} analyzes why it works and when the gain vanishes, Section~\ref{sec:examples}
gives the numerical examples, and Section~\ref{sec:conclusions} concludes and lists directions for further work.

%% file: sections/related_work.tex
\section{Related work}\label{sec:related}

\subsection{Isogeometric analysis}

Isogeometric analysis is a Galerkin method on the spline or NURBS space that already describes the CAD geometry
\cite{hughes2005isogeometric,cottrell2009isogeometric}. Two properties matter here. The
map $\mathbf F:\widehat\Omega=(0,1)^d\to\Omega$ is fixed at the coarsest mesh and unchanged by
refinement, so a point can be placed on $\partial\Omega$ exactly; and $k$-refinement gives $C^{p-1}$
spaces of degree $p$ at roughly one unknown per element, $(n+p)^2$ on an $n\times n$ mesh in two
dimensions against the $(np+1)^2$ of the $C^0$ finite element space of the same degree and mesh,
with established approximation properties, near-optimality against Sobolev $n$-widths and far more
faithful operator spectra \cite{bazilevs2006approximation,evans2009nwidths,cottrell2006vibrations}.
That smoothness is what
makes the Galerkin error amenable to local averaging (Section~\ref{sec:related-superconv}).

\subsection{Physics-informed neural networks and their failure modes}

Collocating a differential operator on a neural ansatz is an old idea, revived as the physics-informed neural
network (PINN) \cite{lagaris1998ann,raissi2019physics}. Its difficulties are well documented.
The eigenvector bias of the neural tangent kernel slows high-frequency learning, and the residual
and boundary loss terms converge at very different rates \cite{wang2021eigenvector,wang2022ntk}. The
residual loss landscape becomes ill-conditioned as a convection or reaction coefficient grows even
when the PINN can represent the solution, so that the failure is optimization rather than
expressivity \cite{krishnapriyan2021failure}. Penalty boundary conditions hold only approximately
and add a
weight to be tuned; the alternative is to build them into the ansatz \cite{sukumar2022exact}.

For linear equations, randomized networks avoid the non-convex training altogether: the hidden-layer parameters
are drawn at random and frozen, and the output weights follow from one linear least-squares problem. The
physics-informed extreme learning machine applies this to linear equations, with one network or, in a distributed
variant, one per cell \cite{dwivedi2020pielm}. Local extreme learning machines assign one such network to each
sub-domain, impose $C^k$ continuity across the sub-domain boundaries and treat nonlinear equations by nonlinear
least squares, which can be trapped in local minima without random perturbations \cite{dong2021locelm}, and the
random feature method pieces local random-feature models together with a partition of unity \cite{chen2022rfm};
both report accuracy on par with classical solvers. All three are stand-alone solvers, with random features and a
least-squares problem coupled over the whole spatial domain.

\subsection{Kolmogorov--Arnold networks}

A KAN replaces the fixed activations of a multilayer perceptron (MLP) by learnable univariate B-splines on the graph
edges \cite{liu2024kan}. Which univariate family those edge functions are drawn from is itself an open question,
and rational ones, Pad\'e approximants and rational Jacobi functions, have been proposed in place of the
splines \cite{aghaei2026rkan}. The transfer to PDEs has mostly substituted a KAN for the MLP inside an existing
physics-informed loss. KINN reports gains on singularities and discontinuous coefficients, together with a loss of
accuracy on geometrically complex domains \cite{wang2025kinn}. MLP and KAN have been compared in a controlled
setting \cite{shukla2024comparison}. The same substitution reaches operator learning and point-cloud regression
\cite{abueidda2025deepokan,kashefi2025kanpointnet}. Two recent schemes move toward variational forms, the second
of them, PG-KINN, with an independent, compactly supported piecewise-polynomial test space
\cite{kim2026weakform,sadr2026pgkinn}. Two of the authors of PG-KINN are authors of the present paper. Outside the
Kolmogorov--Arnold family, PI-GGN trains a graph network on a finite-element Galerkin residual, for forward and
inverse problems alike \cite{gao2022piggn}. Its disk benchmark is one of ours (Section~\ref{sec:inverse}). All of
these use the network as the discretization; we are not aware of work that keeps an isogeometric Galerkin solution
as the substrate and uses KANs only as a local post-process.

\subsection{Classical solver plus learned correction}

Closest in intent is the hybrid FEM-LSSVR scheme, in which element-wise least-squares support vector regression
with higher-order Legendre kernels interpolates, under the PDE, between the nodal values of a low-order finite
element solve and stays consistent with them at the element boundaries \cite{babaei2025superresolution}. Its local
fit is convex and closed-form, and the same principle, a fixed feature map followed by one linear least-squares
solve with the equation in the rows, is what makes the fit of Section~\ref{sec:method} linear. Super-resolution
of PDE fields, learned discretizations and closures, and solver-in-the-loop corrections are trained on reference
data and carry no guarantee on an unseen field
\cite{fukami2019superresolution,fukami2023survey,barsinai2019learning,kochkov2021machine,um2020solver}. Networks
have also predicted the fine-grid fluctuation above a coarse multigrid solution
\cite{margenberg2022multigrid,margenberg2024dnnmg}. Others have corrected a coarse finite element solution, enriched
a Lagrange space and generated element stiffness matrices
\cite{kapustsin2023hybrid,barucq2025enriching,jung2020deeplearned}. The isogeometric proposals so far are surrogates
that reproduce an IGA solution faster: a network produces the spline coefficients, so accuracy is bounded above by
the discretization being imitated \cite{li2023igapignn,moller2026iganets,gasick2023isogeometric}.

\subsection{Post-processing and superconvergence}\label{sec:related-superconv}

The mechanism our method relies on is classical. A local average of a Galerkin solution against a compactly
supported kernel is accurate to $O(h^{2p})$ in the interior of a locally translation-invariant mesh for degree-$p$
elements, although the solution itself is
only $O(h^{p+1})$, because the error superconverges in negative norms; for $p=1$ the two rates coincide and no order
is gained \cite{bramble1977higher}. The same holds for derivatives \cite{thomee1977high}. An early tensor-product
instance and a collection of the theory are also available \cite{douglas1974linfty,wahlbin1995superconvergence}.
The convolution post-processor reappears as the smoothness-increasing accuracy-conserving (SIAC) filter for discontinuous Galerkin fields
\cite{cockburn2003enhanced,ji2014siac}. On the first benchmark of Section~\ref{sec:examples}, a SIAC filter
implemented for comparison is $32$ times cheaper than the hybrid on the reference mesh and reduces the IGA error by
more on an interior box than the hybrid does globally, but globally its $L^2$ error is about ten times that of the
unfiltered IGA solution, because its kernel must be shifted near the boundary. A second classical family recovers a
better field by fitting rather than by convolution, and our per-patch fit belongs to it. The superconvergent patch
recovery (SPR) fits a higher-order polynomial by least squares to the derivatives sampled at superconvergent points of a
nodal patch \cite{zienkiewicz1992spr}. The equilibrium residual and the boundary conditions have been added to that
patch problem \cite{wiberg1993spre,wiberg1994spreb}, the patch polynomials have been blended by the nodal shape
functions \cite{blacker1994conjoint}, and SPR-C imposes equilibrium and boundary conditions as constraints
\cite{rodenas2007sprc}. Polynomial-preserving recovery recovers the solution rather than its gradient
\cite{zhang2005ppr}, and patch recovery has been carried over to adaptive isogeometric analysis as an error
estimator \cite{kumar2017spriga}. Our fit combines these ingredients: a recovery of the solution itself, with the
strong-form residual and the boundary data in its rows, blended by a partition of unity
\cite{melenk1996partition}. What is new is their use as a solution post-process on exact NURBS geometry (boundary
rows on the exact NURBS trace, patches at knot vertices, a periodic seam and a collapsed side identified), the
Kolmogorov--Arnold ridge parametrization, which keeps every operator of Section~\ref{sec:method:operators}
univariate, and the \emph{a posteriori} acceptance filter of Section~\ref{sec:method:safeguard}. The equation in the
rows is what carries the gain: on the first benchmark of Section~\ref{sec:examples} the fit without the physics rows
gains a factor of $4.0$ in $L^2$ where the full fit gains $21$ (Remark~\ref{rem:space}). For splines the superconvergent points are known and
yield optimal-order collocation, so the pointwise structure the classical theory predicts is present in the
isogeometric setting \cite{montardini2017optimal,wang2021greville}.

\subsection{Positioning of the present method}

This is not a neural solver: a classical isogeometric solve is performed first and never discarded, and the
learned object is a family of small local models that post-process it, fitted by one batched linear least-squares
problem, with no optimizer, no learning rate and no local minimum. Relative to FEM-LSSVR, whose principle of a fixed feature
map and one linear system we adopt, we post-process an isogeometric solution on exact NURBS geometry, replace the
Legendre kernel by a Kolmogorov--Arnold ridge model, and fit on overlapping knot-vertex patches blended by a
partition of unity instead of element by element \cite{babaei2025superresolution}. The randomized solvers above
share the linear least-squares fit for linear equations; here the inner functions are fixed directions instead of
random draws, the patch problems are small and independent of one another instead of coupled over the whole domain,
the isogeometric solution enters the rows as data, and nonlinear terms are frozen or linearized at it, so that the
fit stays linear (Section~\ref{sec:method}) \cite{dwivedi2020pielm,dong2021locelm,chen2022rfm}. Unlike the KAN
solvers for PDEs, our KAN never has to represent the solution globally or resolve a singularity on its own. With the
settings of this paper it spans exactly the polynomials of its degree (Remark~\ref{rem:space}), so what we claim is
the safeguarded closed-form post-process built from it, not a gain in expressivity due to the Kolmogorov--Arnold
form. An
\emph{a posteriori} safeguard using no exact solution then accepts a correction, refits it on a smaller support, or
discards it, patch by patch, so where none is accepted the hybrid returns the isogeometric solution. The safeguard
is a computable filter built on the error bound of Section~\ref{sec:mechanism:bounds}; with the settings used
throughout it did not leave the hybrid with a larger error than its input on the benchmarks of
Section~\ref{sec:examples}. Because the mechanism is the
superconvergence of local averages, an $L^2$ gain that grows under refinement requires $p\ge2$: at $p=1$ the two
rates of Section~\ref{sec:related-superconv} coincide, and on the thick ring (Section~\ref{sec:thickring}) the measured
$L^2$ gain at $p=1$ is smaller than at $p=2$ and $p=3$ and decreases between the two meshes computed, from $2.8$ to
$2.0$ (Fig.~\ref{fig:ring}a), while the $H^1$ gain does not decrease. What improves is accuracy per degree of freedom of the Galerkin solve rather than the cost of a
single solve.

%% file: sections/method.tex
\section{The hybrid IGA--KAN method}\label{sec:method}

This section states the method. It starts from an isogeometric solve on the exact geometry, whose Galerkin solution
is $\uh$, and returns a solution on the same geometry that is built like a piecewise multilinear spline on the
isogeometric mesh, except that each nodal value is replaced, wherever a computable test accepts it, by a small local
function fitted in the least-squares sense to the equation, the exact boundary data and $\uh$
(Section~\ref{sec:method:blend}). The method has three steps:
\begin{enumerate}
\item \emph{Isogeometric solve} (Section~\ref{sec:method:iga}): a standard Galerkin solve on the exact NURBS
  geometry gives $\uh$.
\item \emph{Local Kolmogorov--Arnold fit} (Sections~\ref{sec:method:local} and~\ref{sec:method:fit}): on every
  knot-vertex patch, a Kolmogorov--Arnold model of ridge type, with fixed inner functions, is fitted to the strong
  form of the equation, the exact boundary data and $\uh$; all patches are fitted together in one batched linear
  least-squares solve.
\item \emph{Partition-of-unity blend} (Section~\ref{sec:method:blend}): the local models are combined by the
  degree-one hat functions of the isogeometric mesh.
\end{enumerate}
An \emph{a posteriori} safeguard (Section~\ref{sec:method:safeguard}) decides, patch by patch and without the exact
solution, whether a local model is used; a patch whose model is not accepted contributes $\uh$ itself to the blend.
Section~\ref{sec:method:operators} lists the operators to which the construction applies, and
Section~\ref{sec:method:cost} its cost. We call the complete method IGA--KAN or, for short, the hybrid, and its
second part (fit, safeguard and blend, without the solve) the post-process.
Figure~\ref{fig:method_hero} compares the method with classical isogeometric analysis and with a stand-alone
Kolmogorov--Arnold network, and Figure~\ref{fig:method_pipeline} shows its stages.

\begin{figure}[!t]
\centering
\includegraphics[width=\linewidth]{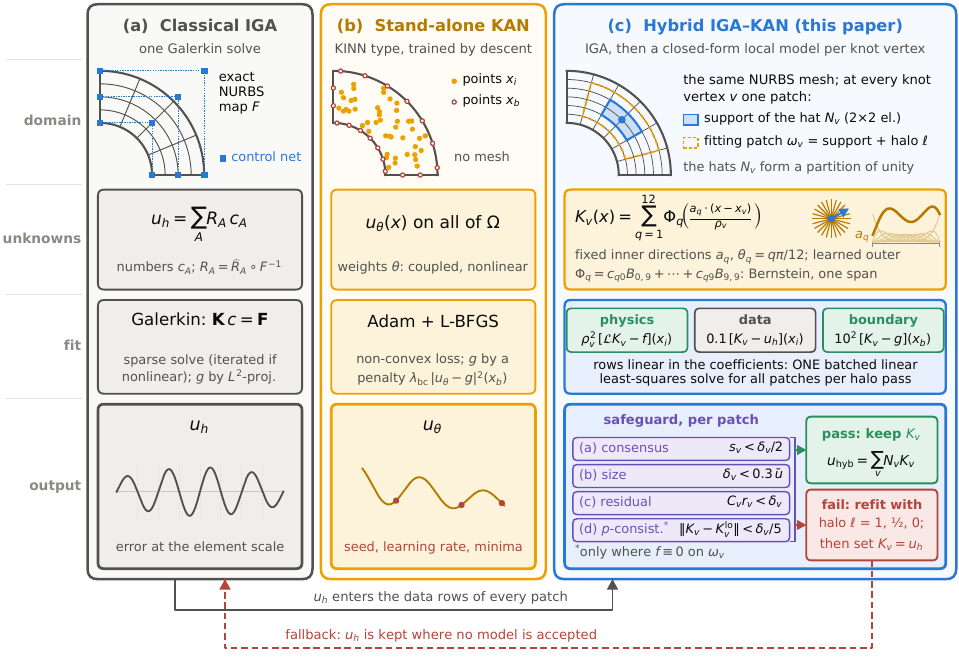}
\caption{What the hybrid adds to a Galerkin solve, read across four rows: what each method sees of the domain,
what it solves for, how the unknowns are found and what it returns. (a) Classical isogeometric analysis, whose
unknowns are the control variables $c_A$. (b) A stand-alone Kolmogorov--Arnold network of the KINN type: one global
$u_\theta$ on sampled points, its weights found by a non-convex descent and its boundary data imposed by a penalty.
(c) The method of this section: on every knot-vertex patch the ridge model \eqref{eq:kan}, with fixed directions
$\ba_q$ and outer functions $\Phi_q$ in the Bernstein basis, fitted by one batched linear least-squares solve to
the physics, data and boundary rows \eqref{eq:rows} and screened by the four tests \eqref{eq:sg-a}--\eqref{eq:sg-d}.
The two lanes below show how the hybrid depends on IGA: $\uh$ enters the data rows of every patch, and a patch that
fails the tests returns $\Kv:=\uh$ (red).}
\label{fig:method_hero}
\end{figure}

\subsection{Model problem and the isogeometric solve}\label{sec:method:iga}

Let $\Omega\subset\mathbb{R}^d$, $d\in\{2,3\}$, be the image of the parametric domain
$\refdom=(0,1)^d$ under a NURBS map $\bF$ of degree $p$ with control points and weights
$(\mathbf{P}_A,w_A)$. Circles, cylinders and the five-petal domain of Section~\ref{sec:flower} are
represented exactly by such a map; no geometry error enters any number reported in this paper.

We seek $u$ with $\mathcal{L}u=f$ in $\Omega$ and $u=g$ on $\Gamma_D\subseteq\partial\Omega$, for
$\mathcal{L}$ one of the operators of Section~\ref{sec:method:operators}. The Galerkin solution
$\uh\in V_h$ is computed in the push-forward NURBS basis $R_A=\widehat R_A\circ\bF^{-1}$, with Gauss
quadrature of at least $p+2$ points per direction, Dirichlet data imposed by $\Ltwo$-projection onto
the trace space and a
Newton or Picard iteration for nonlinear problems. This step is classical; the accuracy of $\uh$ is the
baseline against which every result below is measured.

\subsection{Local Kolmogorov--Arnold ridge models}\label{sec:method:local}

Let $\{\Nv\}_{v\in\mathcal{V}}$ be the degree-one tensor-product B-splines (hat functions) on the
\emph{same} breakpoints as the isogeometric mesh, one per knot vertex $v$, so that
\begin{equation}
  \sum_{v\in\mathcal{V}}\Nv\equiv 1 \quad\text{on }\refdom .
  \label{eq:pou}
\end{equation}
Write $\omega_v$ for the \emph{vertex patch}, or fitting patch, of $v$: the support of $\Nv$, that is, the $2^d$
elements around $v$, enlarged by a halo of $\ell\in\{1,\tfrac12,0\}$ \emph{local} element widths per side, so graded
meshes need no change. Every model is first fitted with $\ell=1$; the smaller halos are used only for the refits of
the safeguard (Section~\ref{sec:method:safeguard}). In what follows ``patch'' without qualification means a vertex patch; the patch of the geometry is called
the NURBS patch. Let
$\xv$ be the centroid of the mapped quadrature points of $\omega_v$ and
$\rhov=\max_i|\bx_i-\xv|$ its radius. On each patch we introduce a Kolmogorov--Arnold model of
\emph{ridge} type,
\begin{equation}
  \boxed{\;\Kv(\bx)=\sum_{q=1}^{Q}\Phi_q\!\left(\frac{\ba_q\cdot(\bx-\xv)}{\rhov}\right),\qquad
  |\ba_q|=1,\;}
  \label{eq:kan}
\end{equation}
with \emph{fixed} linear inner functions and learnable univariate outer B-splines $\Phi_q$ of degree
$\dloc$: the Kolmogorov--Arnold form with the inner layer frozen. The directions are equispaced on the
half circle in two dimensions, $\theta_q=q\pi/Q$, a fundamental domain of $\mathbb{RP}^1$ since $\ba$
and $-\ba$ give the same family, and over the half sphere in three. Throughout the paper $Q=12$ and
$\dloc=9$ with $\Phi_q$ single-span, i.e.\ a Bernstein polynomial of degree $9$ on $[-1.02,1.02]$ with
$\dloc+1=10$ coefficients, so that a local model carries $Q(\dloc+1)=120$ coefficients; in three
dimensions, $\dloc=8$ and $Q=45$ on the interval $[-1.075,1.075]$, whose half-width is $1.02$ divided by the largest
of the seven Gauss abscissae, so that the corners of the patch stay inside it. Figure~\ref{fig:method_patch} shows a patch and its local model.

\paragraph{Role of the Bernstein basis} The Bernstein polynomials carry the learnable part of the model. Each outer
function is expanded as
\begin{equation}
  \Phi_q(z)=\sum_{k=0}^{\dloc}c_{qk}\,B_{k,\dloc}(z),\qquad
  B_{k,\dloc}(z)=\binom{\dloc}{k}\,s^{k}(1-s)^{\dloc-k},\quad s=\frac{z+1.02}{2.04},
  \label{eq:bernstein}
\end{equation}
which is the B-spline basis of a KAN edge with a single knot span; in two dimensions it is evaluated by the same
Cox--de~Boor routine as the isogeometric basis. The $Q(\dloc+1)=120$ coefficients $c_{qk}$ are the only unknowns of the fit
(Section~\ref{sec:method:fit}), and the basis is used for three properties. First, the scaling by $\rhov$ maps the
argument $z_q=\ba_q\cdot(\bx-\xv)/\rhov$ of every sample point into $[-1,1]$, inside the interval of the basis,
which is widened by $2\%$ as a margin. Second, the derivatives of a Bernstein polynomial are Bernstein
polynomials of lower degree whose coefficients are differences of the $c_{qk}$, so $\Phi_q'$ and $\Phi_q''$ in
\eqref{eq:ridgeident} are evaluated exactly as fixed linear combinations of the unknowns; with the inner functions
fixed, every row of \eqref{eq:rows} is therefore linear in the $c_{qk}$. Third, lowering the degree of the basis gives the smaller class of the $p$-consistency test
\eqref{eq:sg-d} ($\dloc=7$, with $Q=10$). The basis adds no approximation power of its own (by the remark below, the local
space is $\mathbb{P}_{\dloc}$ whatever basis represents it), and a multi-span outer spline would replace it by a
piecewise polynomial one.

\begin{figure}[!tbp]
\centering
\includegraphics[width=\linewidth]{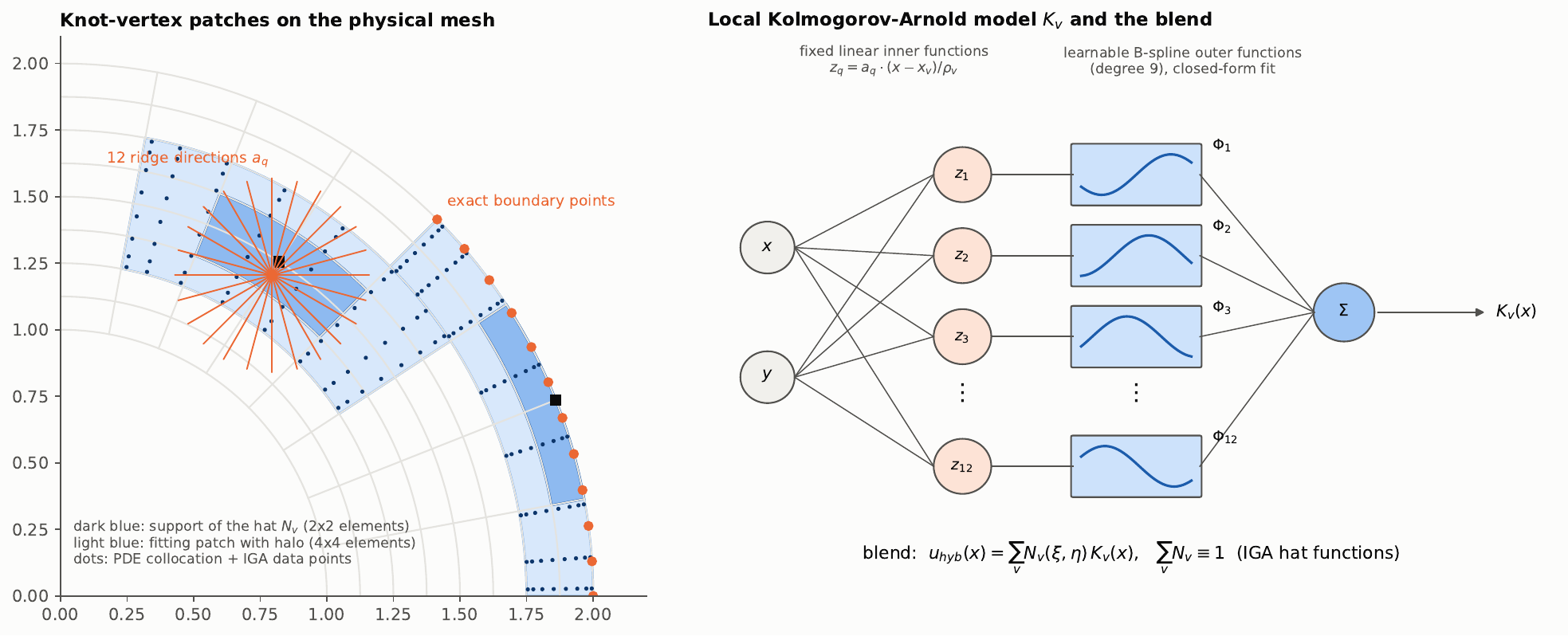}
\caption{Left: the knot-vertex patches on the physical mesh, that is, the support of one hat function $\Nv$
($2\times2$ elements) and its vertex patch $\omega_v$, that support enlarged by the halo, with the collocation and data
points and the exact boundary points marked. Right: the local model \eqref{eq:kan}, with fixed linear
inner functions $z_q=\ba_q\cdot(\bx-\xv)/\rhov$ and learnable degree-$\dloc$ outer B-splines.}
\label{fig:method_patch}
\end{figure}

\begin{remark}[what the local space is]\label{rem:space}
With $Q\ge\dim\mathcal{H}_{\dloc}$ generic directions, where $\mathcal{H}_k$ denotes the homogeneous polynomials
of degree $k$ in $d$ variables ($\dim\mathcal{H}_{\dloc}=\dloc+1$ in two dimensions, $(\dloc+1)(\dloc+2)/2$ in
three), the span of \eqref{eq:kan} is exactly $\mathbb{P}_{\dloc}(\mathbb{R}^d)$, since
$\{(\ba\cdot\mathbf{z})^k\}$ spans $\mathcal{H}_k$. The representation is therefore redundant ($120$ coefficients for a
$55$-dimensional space in two dimensions), and with the settings of this paper the Kolmogorov--Arnold
form is a \emph{parametrization} of $\mathbb{P}_{\dloc}$, not an extension of it. We keep it because it
makes \eqref{eq:ridgeident} univariate, so every operator of Section~\ref{sec:method:operators} reduces
to derivatives of one-dimensional splines, and because a multi-span $\Phi_q$ leaves the polynomial
span, covering local spaces that no polynomial basis reproduces. On the first benchmark of
Section~\ref{sec:examples} a monomial basis for $\mathbb{P}_{9}$ reproduces the hybrid to four digits at a quarter
of the cost, and multi-span outer splines change either error norm by less than $5\%$; without the physics rows
(a recovery of $\uh$ with the boundary data in the same local space) the gain falls from $21$ to $4.0$ in $L^2$
and from $75$ to $3.5$ in $H^1$, and without the boundary rows to $11$ and $35$ (in the error quadrature of this
ablation, one $40$-point Gauss rule per direction over the whole NURBS patch, the full fit gains $21$ and $75$ instead of the
$19$ and $90$ of Table~\ref{tab:ex1}). The gain thus comes from the equation and the exact boundary data in the rows,
not from the choice of basis.
Being a polynomial of degree $\dloc$, the local model is high-order and smooth: it can represent what the solution does inside one
patch, but not a singularity such as a crack tip or a re-entrant corner.
\end{remark}

\subsection{The fit is one linear least-squares problem}\label{sec:method:fit}

The coefficients of $\Phi_q$ follow from three kinds of rows, at the $m=n_s^d$ tensor Gauss points
$\bx_i$ of $\omega_v$ ($n_s=12$ per direction in two dimensions) and the $n_s^{d-1}$ points
$\bx_b$ of each Dirichlet side through it:
\begin{equation}
  \underbrace{\rhov^{2}\bigl[\mathcal{L}\Kv-f\bigr](\bx_i)}_{\text{physics}}
  \;\oplus\;
  \underbrace{\wdata\bigl[\Kv-\uh\bigr](\bx_i)}_{\text{data}}
  \;\oplus\;
  \underbrace{\wbc\bigl[\Kv-g\bigr](\bx_b)}_{\text{boundary}}
  \;\longrightarrow\;\min .
  \label{eq:rows}
\end{equation}
Here $\oplus$ stacks the residuals of the three blocks, at all points $\bx_i$ and $\bx_b$, into one vector, and the
fit minimizes its Euclidean norm, that is, the sum of the squared residuals of all three blocks, over the
coefficients $c_{qk}$ of \eqref{eq:bernstein}.
The boundary rows use the exact NURBS geometry, so the Dirichlet condition is imposed on the true
curved boundary, not on a polygonal approximation. With $\wdata=0.1$ and $\wbc=10^2$, the $\rhov^2$
scaling makes the physics rows dimensionless and puts all three blocks on the same footing, which is
why one pair of weights serves every example in this paper.

With the inner functions fixed, these rows are \emph{linear} in the unknown coefficients, and they are explicit,
because differentiating a ridge function differentiates a univariate spline,
\begin{equation}
  \nabla\Kv=\frac1{\rhov}\sum_q\Phi_q'\,\ba_q,\qquad
  \Delta \Kv=\frac1{\rhov^{2}}\sum_q\Phi_q'' ,
  \label{eq:ridgeident}
\end{equation}
so that $\mathcal{L}=-\kappa\Delta+\mathbf V\cdot\nabla+c_0$ gives the row
\begin{equation}
  \rhov^{2}\,\mathcal{L}\Kv
  =-\kappa\sum_q\Phi''_q+\rhov\sum_q(\mathbf V\cdot\ba_q)\,\Phi'_q+c_0\rhov^{2}\sum_q\Phi_q .
  \label{eq:adrrow}
\end{equation}
Hence \eqref{eq:rows} is an ordinary linear least-squares problem, solved in closed form
(divide-and-conquer SVD: the least-squares matrix is rank-deficient by construction, and the redundancy of the ridge
representation is resolved by the minimum-norm solution; the normal equations are then singular, and regularizing
them biases the fit) and \emph{batched over all patches at once}: no optimizer, learning rate,
initialization, local minimum or back-propagation through the differential operator. Fitting all
$\prod_{k=1}^{d}(n_k+1)$ local models of a mesh ($n_k$ elements in direction $k$) is one call to a
batched least-squares routine.

\subsection{Blending}\label{sec:method:blend}

The local models are combined with the partition of unity \eqref{eq:pou},
\begin{equation*}
  \boxed{\;\uhyb(\bx)=\sum_{v\in\mathcal{V}}\Nv\bigl(\bF^{-1}(\bx)\bigr)\,\Kv(\bx),\qquad
  \nabla\uhyb=\sum_v\bigl(\Kv\nabla\Nv+\Nv\nabla\Kv\bigr).}
\end{equation*}
If every local model is a constant, $\Kv\equiv c_v$, the blend is the piecewise multilinear spline $\sum_vc_v\Nv$
with nodal values $c_v$: the method replaces each nodal value by a local function. If every local model returns the
isogeometric solution, $\Kv=\uh$, the partition of unity \eqref{eq:pou} gives $\sum_v\Nv\uh=\uh$, which makes the
fallback of Section~\ref{sec:method:safeguard} exact rather than approximate.
Where two parametric sides describe the same physical points (a periodic seam) or one collapses to a
point (a polar center), the hat functions are \emph{identified}: those patches share one local model,
fitted on the union of their supports and accepted or not as a unit, so the field is single-valued
at a collapsed side and continuous across a seam, as the isogeometric solution is
(on the flower of Section~\ref{sec:flower} the measured jumps are at the level of round-off).

\subsection{The \emph{a posteriori} safeguard}\label{sec:method:safeguard}

A smooth local model \eqref{eq:kan} need not represent the solution on its patch, for instance at a
singularity, on a coarse mesh or in a boundary layer, and an uncontrolled post-process would then damage an
otherwise good isogeometric solution. Four tests decide the question from computable quantities alone.
For each patch let
\begin{equation}
  \delta_v=\bigl\|\Kv-\uh\bigr\|_{\operatorname{supp}\Nv},\qquad
  s_v=\bigl\|\Kv-\uhyb\bigr\|_{\operatorname{supp}\Nv},\qquad
  r_v=\bigl\|\rhov^2(\mathcal{L}\Kv-f)\bigr\|_{\omega_v},
  \label{eq:safequant}
\end{equation}
all root mean squares over points: $\delta_v$ and $s_v$ over a $3\times3$ grid of check points in each element of
$\operatorname{supp}\Nv$ ($2\times2\times2$ in three dimensions), $r_v$ and the difference in \eqref{eq:sg-d} over
the Gauss points of the fitting patch $\omega_v$. The three quantities are the correction the patch proposes, its
disagreement with the blend of its neighbors, and the residual the fitted model leaves. During the tests $\uhyb$ is
the blend of the current local models of all patches, accepted or not, and it includes $\Kv$ itself, weighted by
$\Nv$, so $s_v$ under-estimates the leave-one-out disagreement and the test is biased toward acceptance. With $\bar u$ the root mean square of $\uh$, a patch is
\emph{accepted} when
\begin{align}
  \text{(a) consensus:}&\quad s_v<\tau\,\delta_v, &&\tau=\tfrac12,\label{eq:sg-a}\\
  \text{(b) size:}&\quad \delta_v<\tau_\delta\,\bar u, &&\tau_\delta=0.3,\label{eq:sg-b}\\
  \text{(c) residual:}&\quad C_v\,r_v<\delta_v, &&\label{eq:sg-c}\\
  \text{(d) $p$-consistency (only where $f\equiv0$ on $\omega_v$):}&
  \quad \bigl\|\Kv-\Kv^{\mathrm{lo}}\bigr\|_{\omega_v}<\tau_p\,\delta_v,
  &&\tau_p=\tfrac15 .\label{eq:sg-d}
\end{align}
Test (a) is the heart of the safeguard: neighboring patches overlap, so a model that has found the
solution agrees with its neighbors \emph{better} than any of them agrees with the isogeometric data,
whereas one that has invented structure does not; (b) excludes corrections so large that the local
model, and not the solver, is the more likely source of the discrepancy. In (c), $C_v$ turns the residual into an error
indicator. It is taken from the maximum-principle bound of Proposition~\ref{prop:local}
(Section~\ref{sec:mechanism}), which controls the error of a local model by its residual and its error on the
boundary of the patch: for $\mathcal{L}=-\kappa\Delta+\mathbf V\cdot\nabla+c_0$ with $c_0\ge0$ in two dimensions the
barrier $w=(\rhov^2-|\bx-\xv|^2)/(4\kappa)$, together with the interior-extremum bound
$\|\mathcal{L}\Kv-f\|_\infty/c_0$ for $c_0>0$, leads to
\begin{equation}
  C_v=\min\left\{\frac{1}{\kappa\beta_v},\ \frac{4}{c_0\rhov^{2}}\right\},\qquad
  \beta_v=1-\frac{|\mathbf V|\rhov}{2\kappa},
  \label{eq:stab}
\end{equation}
where the first entry is replaced by $+\infty$ if $\beta_v\le0$ and the second if $c_0=0$ (for the variable reaction
coefficient of Section~\ref{sec:fk}, $c_0$ is its minimum over the sample points of the patch), in the units of the
$\rhov^{2}$-scaled residual of \eqref{eq:safequant}: four times the constants of Proposition~\ref{prop:local},
normalized so that $C_v=1$ for $-\Delta$. Because it drops the boundary term, replaces the
supremum by a root mean square and uses the two-dimensional barrier constant also in three dimensions, where the
constant of Proposition~\ref{prop:local} is smaller (for the Poisson problem of Section~\ref{sec:thickring}, the only
three-dimensional benchmark, \eqref{eq:stab} gives $C_v=1$), $C_vr_v$ is a computable indicator calibrated on the
bound
\eqref{eq:local-bound} rather than the bound itself; for the space--time heat operator (Section~\ref{sec:heat}) the
test is likewise used with $C_v=1$, a normalization of the scaled residual that is not derived from the parabolic
counterpart of Proposition~\ref{prop:local}. Section~\ref{sec:mechanism:bounds} relates tests (a) and (c) to
the two terms of that bound. Test (d) exists because (c) is blind where $f\equiv0$: every polynomial solution of the homogeneous equation (a harmonic polynomial for the
Laplacian) has zero residual, so a smooth model under-fitting a corner singularity passes (c). It is
applied only there, since where $f\not\equiv0$ the decision rests on (a)--(c), and refits the patch
with a smaller class ($\dloc=7$, $Q=10$), the refit being the $\Kv^{\mathrm{lo}}$ of \eqref{eq:sg-d}: an adequate class gives the same answer up to a fraction
$\tau_p$ of the correction it claims, an inadequate one does not. None of the benchmarks of
Section~\ref{sec:examples} has a singularity. Test (d) is active in two of them, the flower and the heat equation,
where $f\equiv0$; on the flower, where its effect was measured, it is conservative: it excludes patches that pass
(a)--(c) and lowers the gain (Section~\ref{sec:flower}).

A patch failing any test is refitted on a smaller support, $\ell=1\to\tfrac12\to0$, and if it still
fails is set to $\Kv:=\uh$; the accepted set is then re-checked against the \emph{final}
blend in at most three passes, because excluding a patch changes its neighbors' consensus (Algorithm~\ref{alg:hybrid}). The
procedure uses no exact solution and no reference computation, and costs one to three batched
least-squares solves, or twice that where the $p$-consistency test \eqref{eq:sg-d} is active.
Figure~\ref{fig:method_pipeline} draws the procedure.

\begin{figure}[!t]
\centering
\includegraphics[width=\linewidth]{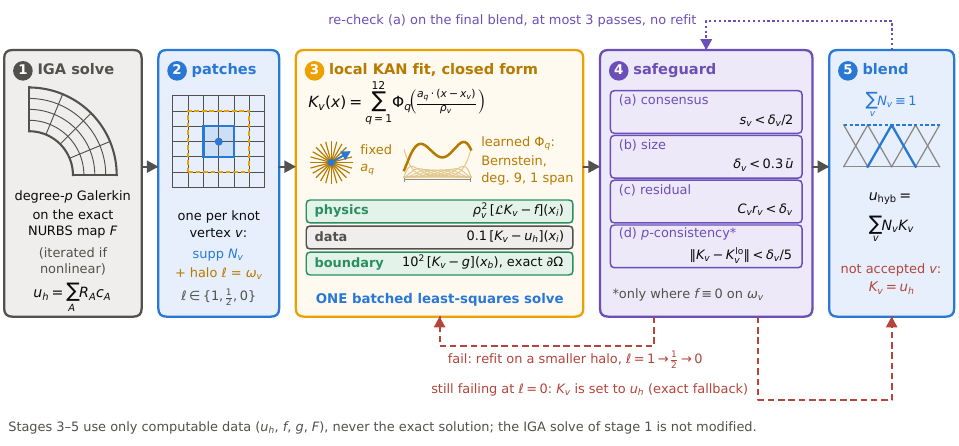}
\caption{The hybrid as five stages (the steps are listed in Algorithm~\ref{alg:hybrid}): (1) the isogeometric solve; (2) one fitting patch
$\omega_v$ per knot vertex, the support of the hat $\Nv$ enlarged by a halo of $\ell$ element widths; (3) the local
model \eqref{eq:kan} with fixed directions and Bernstein outer functions \eqref{eq:bernstein}, fitted to the three
row blocks \eqref{eq:rows} by one batched least-squares solve; (4) the four tests
\eqref{eq:sg-a}--\eqref{eq:sg-d}; (5) the partition-of-unity blend. Red: a patch failing a test is refitted on a
smaller halo and, if it still fails at $\ell=0$, set to $\Kv:=\uh$. Violet: test (a) is re-checked against the
final blend.}
\label{fig:method_pipeline}
\end{figure}

\begin{table}[!t]
\centering
\makeatletter
\refstepcounter{algocf}\label{alg:hybrid}%
\@makecaption{Algorithm~\thealgocf}{Hybrid IGA--KAN super-resolution with the
\emph{a posteriori} safeguard. \textbf{Input:} NURBS patch $\bF$, data $f,g$, degree $p$, mesh;
\textbf{output:} $\uhyb$. Steps 3--9 are the halo loop, run at $\ell=1$, then $\tfrac12$, then
$0$, each pass touching only the patches not yet accepted; the tests (a)--(d) are
\eqref{eq:sg-a}--\eqref{eq:sg-d}.}
\makeatother
\footnotesize
\setlength{\tabcolsep}{4pt}
\renewcommand{\arraystretch}{1.05}
\setlength{\fboxsep}{5pt}
\fbox{%
\begin{minipage}{\dimexpr\linewidth-2\fboxsep-2\fboxrule\relax}%
\begin{tabular}{@{}>{\centering\arraybackslash}p{10pt}>{\raggedright\arraybackslash}p{\dimexpr0.60\dimexpr\linewidth-26pt\relax\relax}>{\raggedright\arraybackslash}p{\dimexpr0.40\dimexpr\linewidth-26pt\relax\relax}@{}}
\multicolumn{3}{@{}c@{}}{\textbf{classical IGA solve $\;\to\;$ local KAN fit $\;\to\;$ four tests $\;\to\;$ partition-of-unity blend}}\\[1pt]
\toprule
\# & what is computed & cost, or what it produces\\
\midrule
\rowcolor{blue!8}\multicolumn{3}{@{}l@{}}{\textbf{Phase 1: the classical solve, unchanged}}\\
1 & $\uh\leftarrow$ \mbox{degree-$p$} Galerkin IGA solve on the exact NURBS geometry $\bF$
  & the baseline; never discarded\\
2 & $\mathcal{T}\leftarrow\emptyset$, the set of accepted vertex patches
  & bookkeeping\\
\rowcolor{green!10}\multicolumn{3}{@{}l@{}}{\textbf{Phase 2, halo loop:} \textbf{for} $\ell\in\{1,\tfrac12,0\}$ \textbf{do} (steps 3--9)}\\
3 & fit every $\Kv$, $v\notin\mathcal{T}$, on its fitting patch $\omega_v$, the support of $\Nv$ with a halo of
    $\ell$ element widths: rows \eqref{eq:rows}, linear by \eqref{eq:ridgeident}
  & \emph{one} batched linear least squares for all patches, $Q(\dloc+1)$ coefficients each ($120$ in two
    dimensions, $405$ in three): no optimizer, no learning rate, no local minimum\\
4 & $\delta_v=\|\Kv-\uh\|$, $s_v=\|\Kv-\uhyb\|$ over $\operatorname{supp}\Nv$ and
    $r_v=\|\rhov^2(\mathcal{L}\Kv-f)\|$ over $\omega_v$ \eqref{eq:safequant}
  & root mean squares at check points and at the Gauss points of the patch; $\uhyb$ is here the blend of all
    current local models; no exact solution enters\\
\rowcolor{orange!8}
5 & test (a) \emph{consensus} \eqref{eq:sg-a}: $s_v<\tau\,\delta_v$, $\tau=\tfrac12$
  & excludes a model that disagrees with its neighbors\\
\rowcolor{orange!8}
6 & test (b) \emph{size} \eqref{eq:sg-b}: $\delta_v<\tau_\delta\,\bar u$, $\tau_\delta=0.3$
  & excludes a correction too large to blame on the solver\\
\rowcolor{orange!8}
7 & test (c) \emph{residual} \eqref{eq:sg-c}: $C_v\,r_v<\delta_v$, with $C_v$ from \eqref{eq:stab}
    ($C_v=1$ in three dimensions and for the heat operator)
  & a computable indicator modeled on the residual term of \eqref{eq:local-bound}\\
\rowcolor{orange!8}
8 & test (d) \emph{$p$-consistency} \eqref{eq:sg-d}, only where $f\equiv0$ on $\omega_v$:
    $\|\Kv-\Kv^{\mathrm{lo}}\|_{\omega_v}<\tau_p\,\delta_v$, $\tau_p=\tfrac15$
  & a refit in the smaller class $\dloc=7$, $Q=10$: one extra batched solve\\
9 & $\mathcal{T}\leftarrow\mathcal{T}\cup\{v:\ \text{(a)--(c) hold, and (d) where applied}\}$;
    the rest go to the next $\ell$
  & \textbf{end for}; the loop stops after $\ell=0$\\
\rowcolor{red!7}\multicolumn{3}{@{}l@{}}{\textbf{Phase 3: fallback}}\\
10 & $\Kv:=\uh$ for every $v\notin\mathcal{T}$
  & $\sum_v\Nv\equiv1$, so where no patch through a point is accepted the blend is exactly $\uh$: an
    exact fallback\\
\rowcolor{violet!8}\multicolumn{3}{@{}l@{}}{\textbf{Phase 4: re-check, at most three passes}}\\
11 & recompute $s_v$ against the blend in which every patch not accepted is $\uh$ and drop from $\mathcal{T}$ every $v$
     failing (a) \eqref{eq:sg-a}; stop when a pass removes nothing
  & no refit: excluding a patch changes its neighbors' consensus\\
\rowcolor{gray!12}\multicolumn{3}{@{}l@{}}{\textbf{Output}}\\
12 & $\uhyb=\sum_v\Nv(\bF^{-1}(\bx))\,\Kv(\bx)$,
     $\nabla\uhyb=\sum_v(\Kv\nabla\Nv+\Nv\nabla\Kv)$ (Section~\ref{sec:method:blend})
   & one to three batched solves in all, twice that where (d) is active; the IGA solve is
     untouched\\
\midrule
\multicolumn{3}{@{}p{\linewidth}@{}}{\textbf{Key point.} Steps 4--11 use only $\uh$, $\Kv$
and the data, never an exact solution or a reference computation. A patch still failing at $\ell=0$
returns $\uh$.}\\
\bottomrule
\end{tabular}%
\end{minipage}}
\end{table}

\subsection{Operators and problem classes}\label{sec:method:operators}

Every row of \eqref{eq:rows} is linear in the coefficients for each of the four classes of operators in the
benchmarks of Section~\ref{sec:examples}, which is what makes the closed-form fit possible in all of them:
\begin{itemize}
\item \emph{Second-order elliptic operators}, $-\kappa\Delta+\mathbf V\cdot\nabla+c_0$ with constant coefficients,
  give the row \eqref{eq:adrrow} (the Poisson and Laplace benchmarks).
\item \emph{Parabolic operators} are treated in space--time form, the last coordinate being time, so $\partial_t$ is
  one more ridge derivative in \eqref{eq:ridgeident}. Time is not stepped: the space--time cylinder is one B-spline
  patch, post-processed in space and time alike (Section~\ref{sec:heat}). Under a time-stepping scheme
  (Sections~\ref{sec:fk} and~\ref{sec:adr}) the local models are instead fitted at each output time on the elliptic
  operator of that level.
\item \emph{Nonlinear terms} are frozen at the classical solution, so that the rows stay linear in $\Kv$: the Bratu
  source becomes the given function $\lambda e^{\uh}$ (Section~\ref{sec:bratu}), and the logistic reaction of the
  Fisher--Kolmogorov equation is linearized by Newton's rule about the discrete solution $\ch$
  (Section~\ref{sec:fk}).
\item \emph{Systems} are super-resolved in the order of their coupling: in Section~\ref{sec:adr} the second field is
  post-processed first and its super-resolved values enter the source term of the first.
\end{itemize}

The identities \eqref{eq:ridgeident} extend to fourth-order operators, through
$\Delta^2\Kv=\rhov^{-4}\sum_q\Phi_q''''$, and, component by component, to linear elasticity. Neither is examined
in this paper: the residual test \eqref{eq:sg-c} rests on a maximum principle, which is not available for
fourth-order operators or for systems such as elasticity, so the safeguard of Section~\ref{sec:method:safeguard} does
not carry over as it stands; counterparts of tests (c) and (d) for these classes are left to future work
(Section~\ref{sec:conclusions}).

\subsection{Cost}\label{sec:method:cost}

The isogeometric solve is unchanged. Per patch the post-process assembles a matrix with $2m+n_b$ rows
($n_b\le 2d\,n_s^{d-1}$ boundary rows) and $Q(\dloc+1)$ columns and solves one least-squares problem, batched over patches
and without sparse factorization; evaluating $\uhyb$ needs the $2^d$ local models whose hats do not
vanish at the point. The fit costs from a fraction of a second to over a minute where the isogeometric solve costs milliseconds to
seconds; solve and fit together take $4$ to about $500$ times the time of the solve on the reference meshes of
Section~\ref{sec:examples} (from $2.1$ against $0.49$~s on the Fisher--Kolmogorov equation to $19.6$ against
$0.039$~s on the flower). What the hybrid offers is therefore accuracy per unknown of the Galerkin solve, above
all in $\Hone$; where only $\Ltwo$ accuracy at the least computing time is wanted, mesh refinement can be the cheaper
route: on the quarter annulus, IGA on $50\times50$ elements matches the $\Ltwo$ error of the hybrid on $16\times16$
in $0.036$~s, against $1.17$~s for the hybrid in the same timing run ($1.18$~s in Table~\ref{tab:ex1}). The fits of different patches
are independent of each other, so the post-process is suited to parallel and GPU execution
(Section~\ref{sec:conclusions}).

%% file: sections/mechanism.tex
\section{Analysis: why a smooth local model can improve the Galerkin solution it is fitted to}\label{sec:mechanism}

The error of the Galerkin solution is orthogonal to the discrete space in the sense of the bilinear form (for a
symmetric coercive problem the solution is the best approximation in the energy norm), so it oscillates on the
element scale, while the \emph{local averages} of $\uh$ are far more accurate: for smooth data a duality argument gives the
negative-norm estimates $\|u-\uh\|_{H^{-s}(\Omega)}\le C\,h^{\,p+1+s}\|u\|_{H^{p+1}(\Omega)}$ for $0\le s\le p-1$,
provided the adjoint problem is $H^{s+2}$-regular, and the Bramble--Schatz post-processes exploit this to recover a
higher order by convolution with a compactly supported kernel \cite{bramble1977higher}.

The post-process is a local average whose kernel is \emph{determined by the equation}.
Of the three blocks of \eqref{eq:rows}, the data rows tie $\Kv$ to $\uh$ only weakly ($\wdata=0.1$): a
least-squares fit of a smooth model to oscillatory data keeps its local averages and discards the
oscillation, so $\Kv$ inherits only the accurate low-frequency content. The physics rows restore the
discarded fine structure (among smooth functions near those averages, $\Kv$ satisfies
$\mathcal{L}\Kv=f$ best on the patch), and the boundary rows fix the remaining freedom where the patch
meets $\partial\Omega$, on the exact geometry. The gain comes from the fact that
$\mathbb{P}_{\dloc}$ \emph{cannot} follow an element-scale oscillation while satisfying the equation; a
model that could would merely reproduce $\uh$. The measured convergence rates are consistent with this picture at
$p=2$: on the two finest mesh pairs the $\Ltwo$ error of the hybrid converges at rates $3.9$ and $4.0$ on the
Fisher--Kolmogorov test on the unit square (Section~\ref{sec:fk}) and $3.8$ and $3.7$ on the quarter annulus
(Section~\ref{sec:poisson}), against $3.0$ for IGA. At $p=3$ the hybrid does not converge faster than IGA on these
mesh pairs ($4.2$ and $3.7$ against $4.3$ and $4.1$ on the quarter annulus, $2.6$ and $3.5$ against $4.0$ on the
Fisher--Kolmogorov test), and the gain there is a factor of $9.7$ to $39$.

\subsection{Error bounds for the blended solution}\label{sec:mechanism:bounds}

Two elementary results bound the error of the hybrid by quantities attached to the local models. The first
concerns the blend: because the hats are non-negative and sum to one, the blend is at every point at least as
accurate as the least accurate of the local models that meet there. It is a pointwise form of the
partition-of-unity estimate of Melenk and Babu\v{s}ka \cite{melenk1996partition}, in which the term carrying the
gradients of the hats involves the disagreement of the local models instead of their errors.

\begin{proposition}[blend]\label{prop:blend}
Let $\uhyb=\sum_v\Nv\Kv$ with $\Nv\ge0$ and $\sum_v\Nv\equiv1$, the hats being read as functions of $\bx$ through
$\bF^{-1}$, let $D_v=\operatorname{supp}(\Nv\circ\bF^{-1})$ and $\mathcal{V}(\bx)=\{v:\bx\in D_v\}$. Then, at every
$\bx$ and, for the gradient, at every $\bx$ in the interior of an element,
\begin{equation*}
  |u-\uhyb|\le\max_{v\in\mathcal{V}(\bx)}|u-\Kv|,\qquad
  |\nabla(u-\uhyb)|\le\max_{v\in\mathcal{V}(\bx)}|\nabla(u-\Kv)|
  +\Bigl(\sum_v|\nabla\Nv|\Bigr)\max_{v,w\in\mathcal{V}(\bx)}\tfrac12\bigl|\Kv-K_w\bigr| .
\end{equation*}
\end{proposition}

\begin{proof}
$u-\uhyb=\sum_v\Nv\,(u-\Kv)$ by the partition of unity, which gives the first bound since the $\Nv$ are
non-negative weights. Differentiating,
$\nabla(u-\uhyb)=\sum_v\Nv\nabla(u-\Kv)+\sum_v\nabla\Nv\,(u-\Kv)$, and because $\sum_v\nabla\Nv=0$ the
second sum is unchanged when any fixed function of $\bx$ is subtracted from every $(u-\Kv)$; taking the
mid-range of the errors $(u-\Kv)(\bx)$, $v\in\mathcal{V}(\bx)$, gives the stated form.
\end{proof}

The second concerns one local model: its error is controlled by its residual, which is computable, and by its
error on the boundary of the patch. It is the classical \emph{a priori} bound of the maximum principle
\cite[Thm.~3.7]{gilbarg2001elliptic}, made explicit on a small patch by a quadratic barrier; its explicit constants
fix $C_v$ in \eqref{eq:stab}.

\begin{proposition}[local model]\label{prop:local}
Let $D\subset\Omega$ be open and contained in the ball of radius $\rho$ about a point $\xv$, let
$\mathcal{L}=-\kappa\Delta+\mathbf V\cdot\nabla+c_0$ with constants $\kappa>0$ and $\mathbf V\in\mathbb{R}^d$ and a
reaction coefficient $c_0\in C(\overline D)$, $c_0\ge0$ (a constant or a function of $\bx$), and let
$u\in C^2(D)\cap C(\overline D)$ satisfy $\mathcal{L}u=f$ in $D$. If
$\beta:=1-|\mathbf V|\rho/(d\kappa)>0$, then every $K\in C^2(D)\cap C(\overline D)$ satisfies
\begin{equation}\label{eq:local-bound}
  \max_{\overline D}|u-K|\;\le\;\max_{\partial D}|u-K|
  +\frac{\rho^{2}}{2d\,\kappa\,\beta}\,\sup_{D}\bigl|\mathcal{L}K-f\bigr| .
\end{equation}
If $\underline c_0:=\inf_D c_0>0$, then, whatever the sign of $\beta$, every such $K$ also satisfies
\begin{equation*}
  \max_{\overline D}|u-K|\;\le\;\max\Bigl\{\max_{\partial D}|u-K|,\;\underline c_0^{-1}\sup_{D}\bigl|\mathcal{L}K-f\bigr|\Bigr\}.
\end{equation*}
\end{proposition}

\begin{proof}
Put $e=K-u$, $r=\mathcal{L}K-f=\mathcal{L}e$, $R=\sup_D|r|$ and $E=\max_{\partial D}|e|$. The function
$w=(\rho^2-|\bx-\xv|^2)\,R/(2d\kappa\beta)$ is non-negative on $\overline D$ and
\[
  \mathcal{L}\bigl(\rho^2-|\bx-\xv|^2\bigr)=2d\kappa-2\,\mathbf V\cdot(\bx-\xv)+c_0\bigl(\rho^2-|\bx-\xv|^2\bigr)
  \;\ge\;2d\kappa-2|\mathbf V|\rho=2d\kappa\beta ,
\]
so $\mathcal{L}w\ge R$ in $D$. Hence $z_\pm=w+E\pm e$ satisfies $\mathcal{L}z_\pm=\mathcal{L}w+c_0E\pm r\ge0$ in $D$
and $z_\pm\ge0$ on $\partial D$, and the weak maximum principle for $\mathcal{L}$ with $c_0\ge0$
\cite[Cor.~3.2]{gilbarg2001elliptic} gives $z_\pm\ge0$ in $D$, that is, $|e|\le E+w\le E+\rho^2R/(2d\kappa\beta)$.
For the second bound, which does not use $\beta$, if $|e|$ attains its maximum over $\overline D$ at an interior point $\bx_0$, say with
$e(\bx_0)>0$, then $\nabla e(\bx_0)=0$ and $\Delta e(\bx_0)\le0$, so
$\underline c_0\,e(\bx_0)\le c_0(\bx_0)\,e(\bx_0)\le\mathcal{L}e(\bx_0)=r(\bx_0)\le R$; the case $e(\bx_0)<0$ is the
same for $-e$, and otherwise the maximum is attained on $\partial D$.
\end{proof}

Combining the two propositions on the interiors of the supports $D_v$, with $\bar\rho_v$ the
radius of $D_v$ about $\xv$ and $\bar\beta_v=1-|\mathbf V|\bar\rho_v/(d\kappa)>0$, gives for a solution
$u\in C^2(\Omega)\cap C(\overline\Omega)$ and every $\bx\in\Omega$
\begin{equation}\label{eq:hybrid-bound}
  |u-\uhyb|(\bx)\;\le\;\max\Bigl\{\;\max_{\substack{v\in\mathcal{T}\\ \bx\in D_v}}
  \Bigl[\max_{\partial D_v}|u-\Kv|+\frac{\bar\rho_v^{\,2}}{2d\,\kappa\,\bar\beta_v}\sup_{D_v}|\mathcal{L}\Kv-f|\Bigr],\;
  |u-\uh|(\bx)\Bigr\},
\end{equation}
where $\mathcal{T}$ is the set of accepted patches and the last term is present only if a patch through $\bx$ was
not accepted. The bound is for the linear operators of Proposition~\ref{prop:local} with the exact right-hand side. Where
the right-hand side of the rows is itself computed from the isogeometric solution, as for a frozen or linearized
nonlinear term, a semi-discrete time derivative or a super-resolved second field (Sections~\ref{sec:bratu},
\ref{sec:fk} and~\ref{sec:adr}), the rows contain an approximation $f_h$ of $f$; for the linearized reaction of
Section~\ref{sec:fk}, whose coefficient $c_0=\alpha(2\ch-1)$ is also taken from the isogeometric solution,
$\mathcal{L}$ is the operator of the rows and $f:=\mathcal{L}u$. Since
$\mathcal{L}(\Kv-u)=(\mathcal{L}\Kv-f_h)+(f_h-f)$, the bound then holds with
$\sup_{D_v}|\mathcal{L}\Kv-f_h|+\sup_{D_v}|f_h-f|$ in place of $\sup_{D_v}|\mathcal{L}\Kv-f|$, and the second
supremum depends on the error of $\uh$ and is not computable. An analogous bound holds for the space--time heat
operator $\partial_t-\kappa\Delta$ in $d_x$ space dimensions: the barrier
$(\bar\rho_v^{\,2}-|\mathbf y-\mathbf y_v|^2)/(2d_x\kappa)$ in the space variables $\mathbf y$ alone
($\mathbf y_v$ the space part of $\xv$), on which the operator equals one, and the weak maximum principle for
parabolic operators \cite[Ch.~3]{protter1984maximum} give \eqref{eq:hybrid-bound} with the parabolic boundary of $D_v$ in
place of $\partial D_v$ and the constant $\bar\rho_v^{\,2}/(2d_x\kappa)$ in front of the residual.

\paragraph{Relation to the safeguard} The bracket in \eqref{eq:hybrid-bound} has two terms, for which tests
\eqref{eq:sg-c} and \eqref{eq:sg-a} of Section~\ref{sec:method:safeguard} are computable surrogates; tests
\eqref{eq:sg-b} and \eqref{eq:sg-d} are not derived from the bound. The residual term is computable where the
right-hand side is exact data: $r_v$ in \eqref{eq:safequant} is the root mean square of the $\rhov^2$-scaled
residual at the sample points, and for the elliptic operators of Proposition~\ref{prop:local} in two dimensions
$C_v\rhov^{2}$, with $C_v$ from \eqref{eq:stab}, is four times the smaller of the two constants of the proposition,
evaluated with $\rho=\rhov$, an entry whose hypothesis fails ($\beta_v\le0$, or $c_0=0$) being counted as $+\infty$
as in \eqref{eq:stab} (in three dimensions and for the heat operator the test is used with $C_v=1$,
Section~\ref{sec:method:safeguard}). The boundary term is
known only on $\partial\Omega$, where $u=g$ and it is the misfit of the boundary rows; in the interior $u$ is
unknown, and the consensus test \eqref{eq:sg-a} is its surrogate, since on their overlap neighboring models are
different estimates of the same solution and their disagreement $s_v$ is the observable trace of their errors.
Proposition~\ref{prop:blend} shows why this quantity matters: in the gradient of the blend, the term amplified by
$\sum_v|\nabla\Nv|=O(1/h)$ depends not on the size of the local errors but only on their \emph{disagreement}. Because a root mean
square replaces the supremum and the consensus replaces the unknown boundary error, passing the tests does not by
itself guarantee that the right-hand side of \eqref{eq:hybrid-bound} is small: the safeguard is a computable filter
built on this bound, and its reliability is assessed on the benchmarks of Section~\ref{sec:examples}. The same
sensitivity suggests why the $\Hone$ gains usually exceed the $\Ltwo$ gains (a factor of $27$ against $5.4$ for $u$
in Section~\ref{sec:adr}): the isogeometric derivative carries the full oscillation; that of the blend, only the
residual disagreement.

\subsection{When the gain vanishes}\label{sec:mechanism:regimes}

The same reasoning identifies the regimes in which the post-process has nothing to gain.
\begin{enumerate}
\item \textbf{Unresolved oscillations.} A discrete solution with the right shape but the wrong phase, the
  typical error of a wave on a marginal mesh, has an error that is not oscillatory noise on an accurate
  average, and a phase-shifted wave satisfies the equation almost exactly, so the physics rows do not
  see it. Section~\ref{sec:heat} shows the transition: at $128$ elements for $25$ periods the safeguard accepts
  $266$ of the $1161$ patches and the hybrid reduces the error only by a factor of $1.28$, against $15.5$ at $256$
  elements, where it accepts $2264$ of $2313$.
\item \textbf{Coarse meshes.} When the patch is not small compared with the features of the solution, the tests
  cannot tell a good local model from a poor one: the safeguard declines, the hybrid returns $\uh$ and the gain is
  $1.00$, on the heat equation up to about three elements per wavelength (Fig.~\ref{fig:heat}f,g) and on the coupled
  system on $8\times8$ elements (Section~\ref{sec:adr}). On that mesh of the coupled system this decision is
  conservative at $p=2$ and $3$, where the unscreened fits would have reduced the $\Ltwo$ error by factors of $3.6$
  to $3.8$, and necessary at $p=4$, where they would have increased it by a factor of $2.4$.
\item \textbf{Non-smooth solutions.} At a re-entrant corner or a crack tip the isogeometric error is a genuine
  local defect, not an element-scale oscillation with accurate averages, and the singular part of the solution is
  not approximated to high order by $\mathbb{P}_{\dloc}$ on a patch that contains the singularity; the residual
  test \eqref{eq:sg-c} is also uninformative there when $f\equiv0$,
  which is the reason for the $p$-consistency test \eqref{eq:sg-d}. None of the benchmarks of
  Section~\ref{sec:examples} has a singularity, so the safeguard has been assessed on smooth solutions only and no
  claim is made here about its behavior at a singularity; enriching the local space for such solutions is left to
  future work (Section~\ref{sec:conclusions}).
\end{enumerate}
With the settings used throughout, the safeguarded hybrid did not return a larger error than the isogeometric
solution it post-processed, beyond round-off, on any mesh of the benchmarks of Section~\ref{sec:examples}, the $24$
three-dimensional runs with local degrees $6$, $8$ and $10$ included (Section~\ref{sec:thickring}).

%% file: sections/numerics.tex
\section{Numerical experiments}\label{sec:examples}

We demonstrate the IGA--KAN hybrid on eight benchmarks: three Poisson problems on curved NURBS domains, one of
them in three dimensions and one also posed as an inverse problem, the Laplace equation on a complex geometry, the nonlinear Bratu equation,
the Fisher--Kolmogorov equation of protein spreading, on the unit square and on a domain fitted to a brain slice
segmented from magnetic resonance images, a high-frequency heat equation and a stiff coupled advection--diffusion--reaction system. Five of them are taken from published work (the disk of PI-GGN
\cite{gao2022piggn}, the flower and the high-frequency heat equation of KINN \cite{wang2025kinn}, the
Fisher--Kolmogorov test of Corti et al.\ \cite{corti2023fk}, and the coupled system of Asmouh and Ostermann
\cite{asmouh2024nurbs}), so that the hybrid can be compared with published numbers where they exist (PI-GGN, KINN,
PG-KINN); the thick ring is the three-dimensional geometry of GeoPDEs \cite{defalco2011}.

The methods compared are the following. \emph{IGA} is the classical isogeometric Galerkin solve $\uh$ on
the exact NURBS geometry \cite{hughes2005isogeometric}. \emph{IGA--KAN} is the hybrid: the same solve, followed on
every knot-vertex patch by one small Kolmogorov--Arnold ridge model fitted in closed form to the strong form of
the equation, to the exact boundary data and to $\uh$, the local models being blended by the hat functions and
screened by an \emph{a posteriori} safeguard. \emph{KAN} is a stand-alone physics-informed Kolmogorov--Arnold
network \cite{liu2024kan} trained on the strong form \cite{raissi2019physics}, $[2,10,10,1]$ with grid $6$, and
\emph{KINN} is the same network built from the layer of KINN \cite{wang2025kinn}; both are trained by Adam followed
by L-BFGS, $1000+200$ iterations, or $2000+400$ on the disk and, for the KAN, on the thick ring. In all these runs the boundary and initial
data are imposed by a penalty of weight $100$, and Adam starts from the learning rate $5\times10^{-3}$, decayed to
$10^{-4}$; these settings were not adjusted to the individual problems. The one exception is the heat equation, on which all networks follow KINN's own protocol
(Section~\ref{sec:heat}). KINN's layer is run on the quarter
annulus, the disk, the Bratu equation, the Fisher--Kolmogorov test on the unit square, the heat equation and the
coupled system; on the thick ring, the flower and the brain slice only the KAN is reported, and on the flower the
errors of KINN are the published ones. The hybrid uses the same settings in
every two-dimensional example: $12$ ridge directions, local degree $9$, $12$ Gauss points per direction per
patch; the three-dimensional example uses $45$ directions, degree $8$ and $7$ points.

\paragraph{Scope of the comparison with the networks} The networks and the hybrid do not solve the same task: each
network solves the equation from scratch, without access to $\uh$, whereas the hybrid post-processes an isogeometric
solution. The network rows of the tables therefore show what a Kolmogorov--Arnold network trained by a fixed
physics-informed protocol delivers on these equations; they do not show the best accuracy these architectures can
reach with tuning to the problem, longer training or adaptive sampling. Within this scope, the comparison is set up
and reported as follows. First, on every benchmark on which we train both, KAN and KINN share the widths, the grid,
the loss, the optimizer and the budget, and differ only in the layer, whose per-edge spline scale gives KINN up to
$10\%$ more parameters. Second, the KINN entries computed here are the median of three training seeds, and so are
the KAN entries on the disk, the flower and the Fisher--Kolmogorov and heat equations and the perceptron of the heat
equation; on the quarter annulus, the thick ring, the Bratu equation and the coupled system the KAN entry is a
single seed; IGA and IGA--KAN are deterministic. Third, where a network has not resolved the solution within the
budget, this is stated with the result: on the Fisher--Kolmogorov equation neither network has converged at the end
of training (Section~\ref{sec:fk}); on the heat equation the perceptron and the KAN built from our layer return
$u\approx0$, and the hybrid is compared with the KAN built from KINN's layer (Section~\ref{sec:heat}); on the
coupled system both networks remain at errors of order $0.4$ (Section~\ref{sec:adr}). Fourth, published results are
quoted as reported, each in its own error measure. Fifth, the factors over the networks quoted in
Sections~\ref{sec:intro} and~\ref{sec:conclusions} are taken against the more accurate network on each benchmark.
They range from $6$ on the disk with $25$ control points, where both networks are $13$ to $15$ times more accurate
than IGA in $\Ltwo$, to $6\times10^{4}$ on the coupled system; against the KAN alone on the Fisher--Kolmogorov
square the factor would be $2\times10^{5}$.

The sizes in the tables are not matched: those of the networks count trainable parameters, those of IGA and
IGA--KAN the Galerkin unknowns, to which the post-process adds $120$ local coefficients per knot vertex ($405$ in
three dimensions), each set determined by its own least-squares problem. The like-for-like comparisons of the hybrid
are therefore with IGA at the same number of Galerkin unknowns, which is its input, on every benchmark (at equal
computing time, mesh refinement can be cheaper in $\Ltwo$, Section~\ref{sec:method:cost}), and, on the quarter
annulus, with two classical post-processors of the same $\uh$: a SIAC filter (Section~\ref{sec:related-superconv})
and a patch recovery of $\uh$ with the boundary data, that is, the same fit without the physics rows
(Remark~\ref{rem:space}).

The accuracy is measured by the relative errors
\begin{equation}\label{eq:err}
  \eLtwo=\frac{\|u-u^\ast\|_{L^2(\Omega)}}{\|u^\ast\|_{L^2(\Omega)}},\qquad
  \eHone=\frac{\|u-u^\ast\|_{H^1(\Omega)}}{\|u^\ast\|_{H^1(\Omega)}},
\end{equation}
with $u^\ast$ the exact solution and $u$, in \eqref{eq:err} and in the captions of this section, the solution
computed by the method in question ($\uh$, $\uhyb$ or a network), both evaluated by Gauss quadrature on the exact
geometry; the tables report these relative errors. For each benchmark the methods are compared on one
\emph{reference mesh}, the mesh named in the caption of its table; results on the other meshes computed are quoted in
the text and in the figures. The error maps of the figures show instead the pointwise absolute error $|u-u^\ast|$, each panel on
its own color bar, and the number at the upper end of a bar is the \emph{maximum pointwise absolute error}
$\max|u-u^\ast|$ over the plotted points of that panel. It is neither relative nor averaged over the domain, so it
is not comparable with the numbers in the tables. Network architectures other than $[2,10,10,1]$ are given in the
tables. The size listed for IGA and
IGA--KAN is that of the Galerkin solve: `unknowns' are the coefficients not fixed by the Dirichlet data (quarter
annulus, thick ring, flower, Bratu equation), and `coefficients' or `control points' are all coefficients of the
spline space, those fixed by the Dirichlet data included (the other examples). The post-process adds one local model per knot vertex, with $120$
coefficients in two dimensions and $405$ in three, each determined by its own least-squares problem. All computations are in double precision on one CPU
(Apple M4); the times of IGA and IGA--KAN are medians of three runs, and the time of IGA--KAN includes the solve.

\subsection{Poisson equation}

\subsubsection{Quarter annulus}\label{sec:poisson}
We start from the Poisson equation on a quarter annulus, the standard curved domain of isogeometric analysis
\cite{hughes2005isogeometric},
\begin{equation}\label{eq:p1}
  -\Delta u=f\ \text{ in }\Omega=\{(x,y):\,1<x^2+y^2<4,\ x>0,\ y>0\},\qquad u=0\ \text{ on }\partial\Omega ,
\end{equation}
with $f$ chosen such that $u^\ast(x,y)=(x^2+y^2-1)(x^2+y^2-4)\sin x\sin y$. The domain is one NURBS patch whose arcs
are exact. On the reference mesh, $p=2$ on $16\times16$ elements ($256$ unknowns), the hybrid lowers the
error of the IGA solution by a factor of $19$ in $L^2$ and of $90$ in $H^1$, in $1.2$~s, and is more than
two orders of magnitude more accurate than both networks (Table~\ref{tab:ex1}). Figure~\ref{fig:ex1} shows the three
kinds of error side by side: that of the network is irregular and largest along the boundary, where its data are
imposed by a penalty; that of IGA follows the element pattern of the spline space; and that of the hybrid is more
than an order of magnitude smaller, what remains lying mostly next to the outer arc.

\begin{figure}[!htbp]
\centering
\includegraphics[width=\linewidth]{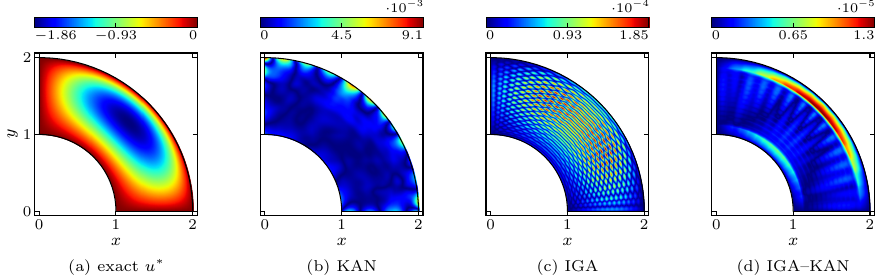}
\caption{Poisson equation on the quarter annulus, $p=2$ on $16\times16$ elements: (a) the exact solution; (b)--(d)
the pointwise absolute errors of the KAN, of IGA and of IGA--KAN, each on its own color bar. The upper end of
each error color bar is the maximum pointwise absolute error $\max|u-u^\ast|$ of that panel; the relative errors
$\eLtwo$ and $\eHone$ are given in Table~\ref{tab:ex1}.}
\label{fig:ex1}
\end{figure}

\begin{table}[!htbp]
\centering
\caption{Poisson equation on the quarter annulus: comparison of the methods on the reference mesh ($p=2$,
$16\times16$).}
\label{tab:ex1}
\small
\begin{tabular}{@{}lrccr@{}}
\toprule
method & size & $\eLtwo$ & $\eHone$ & time [s] \\
\midrule
KAN     & 1300 parameters & $\sci{1.22}{-3}$ & $\sci{4.75}{-3}$ & 49 \\
KINN    & 1430 parameters & $\sci{1.50}{-3}$ & $\sci{4.77}{-3}$ & 538 \\
IGA     & 256 unknowns    & $\sci{6.57}{-5}$ & $\sci{1.51}{-3}$ & 0.006 \\
IGA--KAN & 256 unknowns   & $\mathbf{\sci{3.46}{-6}}$ & $\mathbf{\sci{1.67}{-5}}$ & 1.18 \\
\midrule
\multicolumn{2}{@{}l}{gain of IGA--KAN over IGA} & $19.0\times$ & $90.0\times$ & \\
\multicolumn{2}{@{}l}{gain of IGA--KAN over KINN} & $434\times$ & $285\times$ & \\
\bottomrule
\end{tabular}
\end{table}

\subsubsection{Unit disk: forward and inverse solution}\label{sec:inverse}
Next we solve the disk benchmark of PI-GGN \cite{gao2022piggn} (their Sec.~3.1),
\begin{equation}\label{eq:disk}
  -\Delta u=f\ \text{ in }\Omega=\{\bx:\,|\bx|<1\},\qquad u=0\ \text{ on }\partial\Omega,\qquad
  u^\ast=\tfrac14\,(1-x^2-y^2)\ \text{ for } f=1 .
\end{equation}
The disk is a single quadratic NURBS patch with an exact circle. In the \emph{forward} problem $f=1$ is given. In
the \emph{inverse} problem the constant source $f^\ast=2$ is unknown and $u$ is observed at one point,
$\bx_s=(0.46,-0.46)$, read from Fig.~3a of PI-GGN. Because $u$ is linear in $f$, IGA and IGA--KAN identify the source
by one division, $\hat f=u^\ast_{f^\ast}(\bx_s)/u^h_1(\bx_s)$, the observed value divided by the value at $\bx_s$ of
the solution computed with $f=1$; KINN trains the network together with $f$, with a
penalty $\lambda=1000$ on the datum or with the datum enforced exactly, as PI-GGN does. On $25$ control points, the
size of PI-GGN's graph, the hybrid is the most accurate of all methods in both problems (Fig.~\ref{fig:disk-fields},
Table~\ref{tab:disk}). The solution is a quadratic polynomial, which the local models reproduce exactly away from the
boundary, so this benchmark is a favorable case for the hybrid. PI-GGN's error is a nodal norm and is quoted as
reported. For $20$ observation points drawn at random in the disk the ratio of the median source errors of IGA and
IGA--KAN is $51$, against $167$ at $\bx_s$. With a relative Gaussian perturbation of standard deviation $\sigma$
of the observation ($400$ realizations), the root-mean-square source error of IGA--KAN is $153$, $38$, $4.4$ and
$1.09$ times smaller than that of IGA for $\sigma=10^{-5}$, $10^{-4}$, $10^{-3}$ and $10^{-2}$: the advantage shrinks
once the noise exceeds the source error of the hybrid and is lost once it exceeds that of IGA, $\sci{4.14}{-3}$.

Figure~\ref{fig:disk-curves} shows that these results are not a property of the $25$-point mesh. In the forward
problem (a) the hybrid stays below IGA on every mesh at both degrees. In the inverse problem (b) the source error of
KINN falls below PI-GGN's bound of $1\%$ only in the L-BFGS stage, after $2000$ Adam iterations, and ends between the
one-division results of IGA and IGA--KAN. Under refinement (c) the one-division source of IGA and IGA--KAN converges
with the mesh, while that of the network, shown as a band, does not depend on it.

\begin{figure}[!htbp]
\centering
\includegraphics[width=\linewidth]{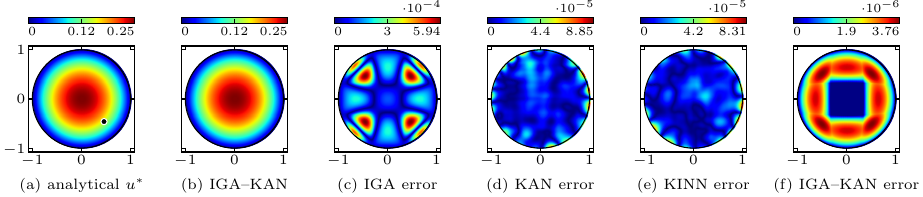}
\caption{Unit disk, forward problem, $p=2$ on $3\times3$ elements ($25$ control points): (a) the analytical solution
with the observation point of the inverse problem; (b) the IGA--KAN solution; (c)--(f) the pointwise errors of IGA,
the KAN, KINN and IGA--KAN, each on its own color bar (color map and layout as in PI-GGN's Figs.~2--3). The
upper end of each error color bar is the maximum pointwise absolute error $\max|u-u^\ast|$ of that panel; the
relative errors $\eLtwo$ and $\eHone$ are given in Table~\ref{tab:disk}.}
\label{fig:disk-fields}
\end{figure}

\begin{figure}[!htbp]
\centering
\includegraphics[width=\linewidth]{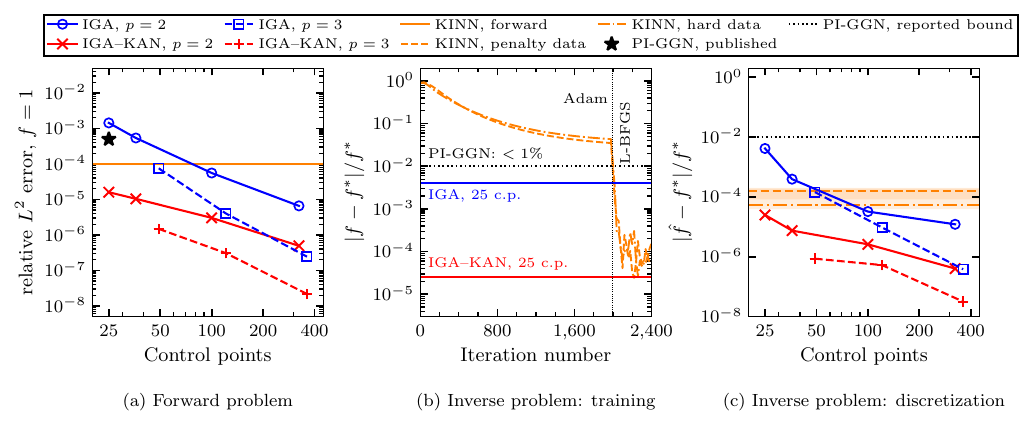}
\caption{Unit disk: (a) relative $L^2$ error $\eLtwo$ of the forward problem against the number of control points, for IGA
and IGA--KAN at $p=2,3$, KINN and PI-GGN's published value (a nodal error, as in Table~\ref{tab:disk}, not $\eLtwo$); (b) relative error of the inferred source during the
training of KINN (penalty and hard data), against the one-division results of IGA and IGA--KAN on $25$ control
points and PI-GGN's reported bound; (c) relative error of the inferred source against the number of control
points, the band marking the two KINN results.}
\label{fig:disk-curves}
\end{figure}

\begin{table}[!htbp]
\centering
\caption{Unit disk: comparison of the methods on $25$ control points ($p=2$, $3\times3$), the size of PI-GGN's graph.
Forward: relative errors of $u$ for $f=1$. Inverse: relative error of the identified source $|\hat f-f^\ast|/f^\ast$.
The times of IGA and IGA--KAN are those of the forward problem; identifying the source adds $0.002$~s to each.}
\label{tab:disk}
\small
\begin{tabular}{@{}lrcccr@{}}
\toprule
method & size & $\eLtwo$ (forward) & $\eHone$ (forward) & source error (inverse) & time [s] \\
\midrule
PI-GGN \cite{gao2022piggn} & 25 nodes & $\sci{5}{-4}$ (nodal) & --- & $<1\%$ & --- \\
KAN                 & 1300 parameters & $\sci{1.12}{-4}$ & $\sci{5.02}{-4}$ & --- & 105 \\
KINN, forward       & 1430 parameters & $\sci{9.75}{-5}$ & $\sci{4.17}{-4}$ & ---               & 196 \\
KINN, penalty data  & 1430 parameters & ---               & ---               & $\sci{1.55}{-4}$ & 200 \\
KINN, hard data     & 1430 parameters & ---               & ---               & $\sci{5.39}{-5}$ & 205 \\
IGA                 & 25 control points & $\sci{1.43}{-3}$ & $\sci{7.07}{-3}$ & $\sci{4.14}{-3}$ & 0.006 \\
IGA--KAN            & 25 control points & $\mathbf{\sci{1.59}{-5}}$ & $\mathbf{\sci{3.22}{-5}}$ & $\mathbf{\sci{2.48}{-5}}$ & 0.08 \\
\bottomrule
\end{tabular}
\end{table}

\subsubsection{Thick ring: a three-dimensional problem}\label{sec:thickring}
To test the method in three dimensions we solve the Poisson equation on the thick ring of GeoPDEs
\cite{defalco2011,vazquez2016},
\begin{equation}\label{eq:ring}
\begin{gathered}
  -\Delta u=f\ \text{ in }\Omega=\{(x,y,z):\,1<x^2+y^2<4,\ x>0,\ y>0,\ 0<z<1\},\qquad u=u^\ast\ \text{ on }\partial\Omega,\\
  u^\ast=\tfrac1{10}\,(r^2-1)(4-r^2)\,z(1-z),
\end{gathered}
\end{equation}
with $r^2=x^2+y^2$. The domain is one trivariate NURBS volume with exact arcs. In three dimensions the local space
needs more directions: the hybrid uses local degree $8$, $45$ ridge directions and $7$ Gauss points per direction,
and omits the $p$-consistency test since $f\neq0$. On the reference mesh, $p=2$ on $16^3$ elements ($4096$
unknowns), it reduces the error of the IGA solution by a factor of $29$ in $L^2$ and of $92$ in $H^1$ and is more than
three orders of magnitude more accurate than the KAN (Table~\ref{tab:ring}).
The solution vanishes on the two cylindrical faces and on $z=0$ and $z=1$, but not on the flat faces $x=0$ and
$y=0$, where the Dirichlet datum is $u^\ast$; Fig.~\ref{fig:ring-fields} shows the errors on two interior sections:
the IGA error is the element-scale pattern of the spline space across the wall, and the hybrid lowers its maximum on
the sections by factors of $12$ and $14$, what remains lying next to the two cylindrical faces.

\begin{figure}[!htbp]
\centering
\includegraphics[width=0.795\linewidth]{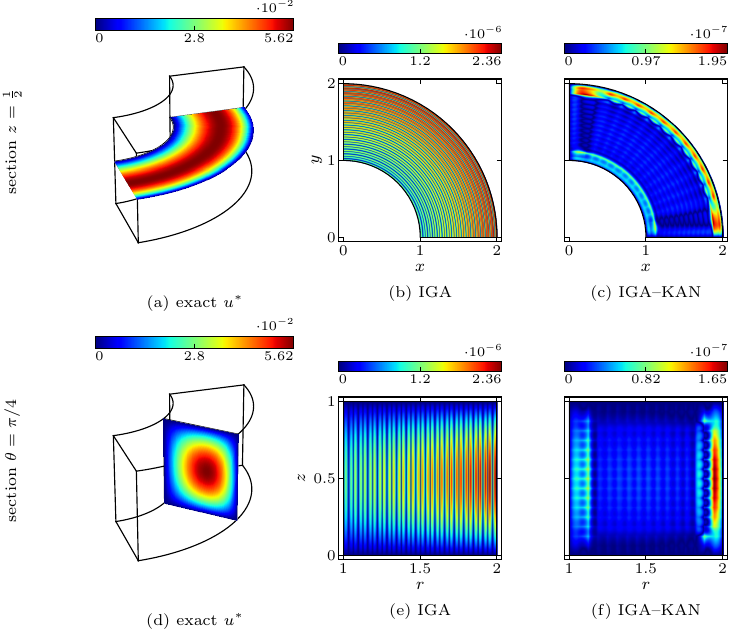}
\caption{Thick ring, $p=2$ on $16^3$ elements, on two interior sections: the mid-height section $z=\tfrac12$ (a--c)
and the diagonal section $\theta=\pi/4$ (d--f). (a), (d) The section inside the ring, colored by the exact solution;
(b), (c), (e), (f) the pointwise absolute errors $|u-u^\ast|$ of IGA and IGA--KAN on that section, the diagonal
section drawn in the $(r,z)$ plane. The solution vanishes on four of the six faces, so the errors are shown on
interior sections. The upper end of each error color bar is the maximum pointwise absolute error of that panel; the
relative errors are given in Table~\ref{tab:ring}.}
\label{fig:ring-fields}
\end{figure}

Figure~\ref{fig:ring} shows the convergence and the role of the local degree. With local degree $8$ (a) the hybrid
gains on all eight meshes, the safeguard accepting $98$ to $100\%$ of the patches: in $L^2$ by factors of $18$ to $37$
at $p=2$, of $3.8$ to $70$ at $p=3$, where the gain is largest on the coarsest mesh, and of $2.0$ to $2.8$ at $p=1$,
where it decreases from the coarser to the finer mesh (at $p=1$ local averaging gains no order,
Section~\ref{sec:related-superconv}); the $H^1$ gains are $64$ to $112$, $18$ to
$102$ and $12$ to $13$. Panel (b) shows the dependence on the local degree, with $(\dloc+1)(\dloc+2)/2=28$, $45$ and $66$ directions for
degrees $6$, $8$ and $10$. Degree $6$ gives gains of $2.0$ to $289$ (in three of its runs, on $16^3$ elements, the
SVD driver did not converge for one patch, which was re-solved after column scaling).
With degree $10$ on the same seven sample points per direction the safeguard
accepts as few as $10\%$ of the patches, and the gain falls to $1.04$ to $14$; panel (c) shows that large gains need
almost all patches accepted. In none of the $24$ runs is the hybrid less accurate than IGA.

The exact solution in \eqref{eq:ring} is a polynomial of degree six and lies in the local space, so, like the disk,
this benchmark is a favorable case. For the non-polynomial test function $u^\ast=e^{x}\sin(xy)\cos z$ of the
thick-ring example distributed with GeoPDEs \cite{defalco2011,vazquez2016}, with Dirichlet data on all six faces and
the same settings, the hybrid gains on every mesh, by smaller factors: on the reference mesh by $3.7$ in $L^2$ and
$5.8$ in $H^1$, and over all meshes by $3.7$ to $31$ and $4.9$ to $21$ at $p=2$, by $2.7$ to $15$ and $2.8$ to $11$ at
$p=3$ and by $2.2$ to $2.5$ and $2.0$ to $2.5$ at $p=1$, the largest $L^2$ factors on the coarsest mesh; the
safeguard accepts $65$ to $96\%$ of the patches.

\begin{figure}[!htbp]
\centering
\includegraphics[width=\linewidth]{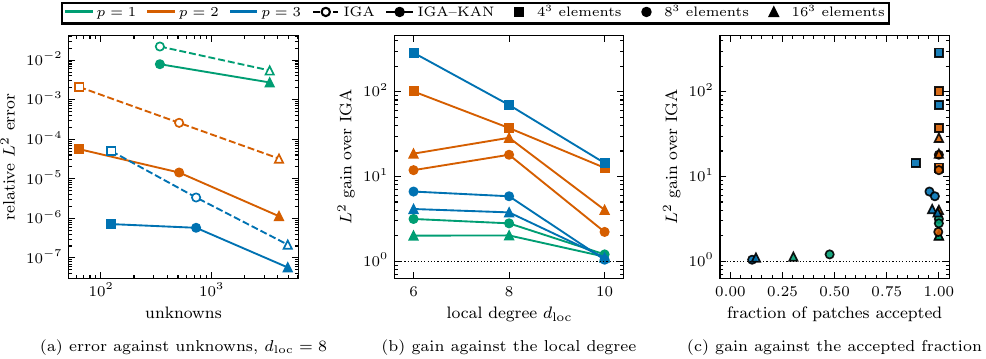}
\caption{Thick ring. (a) Relative $L^2$ error $\eLtwo$ against the number of unknowns for IGA (dashed, open markers) and
IGA--KAN (solid, filled markers) at $p=1,2,3$, local degree $\dloc=8$; (b) $L^2$ gain of IGA--KAN over IGA against the
local degree $\dloc$, one line per degree $p$ and mesh; (c) $L^2$ gain against the fraction of patches the safeguard
accepts, for all $24$ runs. Color: degree $p$; marker: mesh.}
\label{fig:ring}
\end{figure}

\begin{table}[!htbp]
\centering
\caption{Thick ring: comparison of the methods on the reference mesh ($p=2$, $16\times16\times16$).}
\label{tab:ring}
\small
\begin{tabular}{@{}lrccr@{}}
\toprule
method & size & $\eLtwo$ & $\eHone$ & time [s] \\
\midrule
KAN $[3,10,10,1]$ & 1400 parameters & $\sci{5.78}{-3}$ & $\sci{1.82}{-2}$ & 14\,299 \\
IGA      & 4096 unknowns & $\sci{3.24}{-5}$ & $\sci{7.16}{-4}$ & 6.3 \\
IGA--KAN & 4096 unknowns & $\mathbf{\sci{1.13}{-6}}$ & $\mathbf{\sci{7.76}{-6}}$ & 80.9 \\
\midrule
\multicolumn{2}{@{}l}{gain of IGA--KAN over IGA} & $28.7\times$ & $92.1\times$ & \\
\multicolumn{2}{@{}l}{gain of IGA--KAN over KAN} & $5100\times$ & $2340\times$ & \\
\bottomrule
\end{tabular}
\end{table}

\subsection{Laplace equation on a complex geometry}\label{sec:flower}
We then consider the complex-geometry benchmark of KINN \cite{wang2025kinn} (their Sec.~4.3),
\begin{equation}\label{eq:flower}
  \Delta u=0\ \text{ in }\Omega,\qquad u=u^\ast\ \text{ on }\partial\Omega,\qquad u^\ast=\sin x\sinh y+\cos x\cosh y ,
\end{equation}
on the five-petal flower bounded by five unit semicircles. The flower is one NURBS patch whose boundary consists of
ten rational quadratic arcs, exact to machine precision \cite{piegl1997nurbs}. KINN reports $\sci{4.73}{-3}$ for its
boundary-integral KAN and $\sci{2.06}{-4}$ for the same method with a perceptron, and PG-KINN
\cite{sadr2026pgkinn} reports the same value as the former, $\sci{4.73}{-3}$, all three in KINN's own discrete error measure (defined in
Table~\ref{tab:flower}). On the reference mesh, $p=2$ with $80\times16$ elements, IGA--KAN reduces the error of the
IGA solution by a factor of $4.2$ and, in KINN's measure, is $95$ times more accurate than the published KINN and
PG-KINN results (Fig.~\ref{fig:flower}, Table~\ref{tab:flower}). The NURBS patch is polar: its eighty angular elements
close on a periodic seam and its innermost side collapses to the center, where the local models are identified
(Section~\ref{sec:method:blend}). Since $f\equiv0$, the $p$-consistency test (d) is active; with it the safeguard accepts $904$ of the
$1377$ patches although the solution is smooth, against $1296$ without it, where the gain would be $14$ in $L^2$ and
$44$ in $H^1$. The
figure shows how the errors are distributed: that of the KAN is spread over the whole flower, that of IGA follows the
radial element lines and peaks at the tips of the petals, and the hybrid removes most of it, the remainder lying along a
few radial lines.

\begin{figure}[!htbp]
\centering
\includegraphics[width=\linewidth]{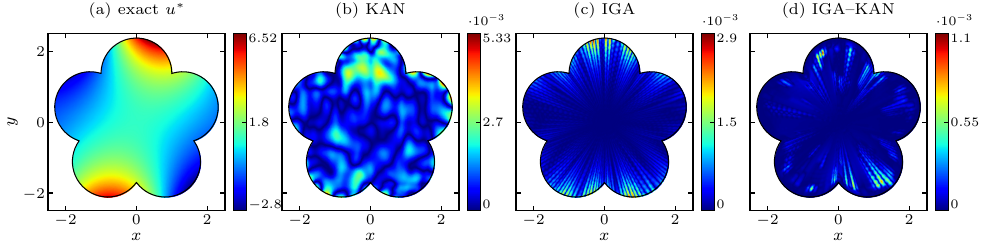}
\caption{Flower, $p=2$ on $80\times16$ elements: (a) the exact solution; (b)--(d) the pointwise absolute errors of the
KAN, of IGA and of IGA--KAN, each on its own color bar (color map as in KINN's Fig.~25). The upper end of each
error color bar is the maximum pointwise absolute error $\max|u-u^\ast|$ of that panel, $\sci{1.1}{-3}$ for
IGA--KAN; Table~\ref{tab:flower} gives the relative errors, $\eLtwo=\sci{4.80}{-5}$ for IGA--KAN.}
\label{fig:flower}
\end{figure}

\begin{table}[!htbp]
\centering
\caption{Flower: comparison of the methods on the reference mesh ($p=2$, $80\times16$ elements, $1361$ unknowns).
Columns 3 and 4 are the relative errors $\eLtwo$ and $\eHone$ of \eqref{eq:err}, computed by Gauss quadrature over the
whole flower. Column 5 is KINN's error measure, $\|u-u^\ast\|_2/\|u^\ast\|_2$ with discrete $2$-norms over KINN's
$52\,264$ evaluation points (a uniform $300\times300$ grid restricted to the flower, without a band about $0.03$ wide
along the boundary); it is the only measure in which the published boundary-integral neural network (BINN; KINN's Table~6) and PG-KINN
(its Table~5) numbers exist, so
only this column compares all methods. A dash means \emph{not reported}.}
\label{tab:flower}
\footnotesize
\begin{tabular}{@{}lrcccr@{}}
\toprule
 & & \multicolumn{2}{c}{our measures} & KINN's measure & \\
\cmidrule(lr){3-4}\cmidrule(lr){5-5}
method & size & $\eLtwo$ & $\eHone$ & $\|u-u^\ast\|_2/\|u^\ast\|_2$ & time [s] \\
\midrule
\multicolumn{6}{@{}l}{\emph{published}}\\
KINN-BINN (KAN) \cite{wang2025kinn}      & 975 parameters  & --- & --- & $\sci{4.73}{-3}$ & --- \\
BINN (MLP) \cite{wang2025kinn}           & 2911 parameters & --- & --- & $\sci{2.06}{-4}$ & --- \\
PG-KINN \cite{sadr2026pgkinn}            & $[2,5,5,5,1]$   & --- & --- & $\sci{4.73}{-3}$ & --- \\
\midrule
\multicolumn{6}{@{}l}{\emph{this work}}\\
KAN      & 1300 parameters & $\sci{6.24}{-4}$ & $\sci{2.40}{-3}$ & $\sci{6.24}{-4}$ & 691 \\
IGA      & 1361 unknowns   & $\sci{2.00}{-4}$ & $\sci{4.34}{-3}$ & $\sci{1.94}{-4}$ & 0.04 \\
IGA--KAN & 1361 unknowns   & $\mathbf{\sci{4.80}{-5}}$ & $\mathbf{\sci{1.05}{-3}}$ & $\mathbf{\sci{4.96}{-5}}$ & 19.6 \\
\bottomrule
\end{tabular}
\end{table}

\subsection{Nonlinear Bratu equation}\label{sec:bratu}
The Bratu equation is the standard nonlinear benchmark of isogeometric analysis \cite{dalcin2016petiga},
\begin{equation}\label{eq:bratu}
  -\Delta u=\lambda e^{u}\ \text{ in }\Omega=(0,1)^2,\qquad u=u^\ast\ \text{ on }\partial\Omega,\qquad
  u^\ast(x,y)=-2\ln\frac{\cosh\bigl((x-\tfrac12)\theta/2\bigr)}{\cosh(\theta/4)},
\end{equation}
with $\lambda=1$ and $\theta$ the lower-branch root of $\theta=\sqrt{2\lambda}\cosh(\theta/4)$. The isogeometric solve
uses Picard iteration, each step one linear Galerkin solve. The hybrid keeps the local fit linear by freezing the source
at $\lambda e^{\uh}$, and reduces the error of the IGA solution by a factor of $18$ in $L^2$ and of $84$ in $H^1$ on the
reference mesh; KINN's layer is the better of the two networks here, but remains three orders of magnitude behind
the hybrid in $L^2$ (Fig.~\ref{fig:bratu}, Table~\ref{tab:bratu}). The solution depends on $x$ only. The figure shows that
the IGA error is the element-scale oscillation of the spline space across the whole square, and that the error of
the hybrid is more than ten times smaller at its maximum and lies mainly in the patches next to the sides $x=0$ and
$x=1$.

\begin{figure}[!htbp]
\centering
\includegraphics[width=0.72\linewidth]{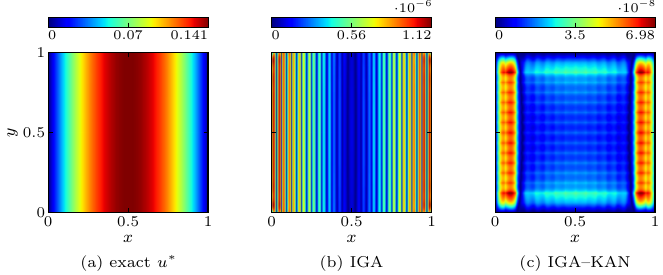}
\caption{Bratu equation, $\lambda=1$, $p=2$ on $16\times16$ elements: (a) the exact solution; (b), (c) the pointwise
absolute errors of IGA and of IGA--KAN, each on its own color bar. The upper end of each error color bar is the
maximum pointwise absolute error $\max|u-u^\ast|$ of that panel; the relative errors $\eLtwo$ and $\eHone$ are given
in Table~\ref{tab:bratu}. The network runs of that table stored their errors but not the networks, so they have no
panel.}
\label{fig:bratu}
\end{figure}

\begin{table}[!htbp]
\centering
\caption{Bratu equation, $\lambda=1$: comparison of the methods on the reference mesh ($p=2$, $16\times16$, $256$
unknowns). IGA takes $11$ Picard iterations.}
\label{tab:bratu}
\small
\begin{tabular}{@{}lrccr@{}}
\toprule
method & size & $\eLtwo$ & $\eHone$ & time [s] \\
\midrule
KAN      & 1300 parameters & $\sci{6.94}{-4}$ & $\sci{3.50}{-3}$ & 690 \\
KINN     & 1430 parameters & $\sci{2.95}{-4}$ & $\sci{9.96}{-4}$ & 1618 \\
IGA      & 256 unknowns    & $\sci{4.74}{-6}$ & $\sci{1.48}{-4}$ & 0.22 \\
IGA--KAN & 256 unknowns    & $\mathbf{\sci{2.60}{-7}}$ & $\mathbf{\sci{1.75}{-6}}$ & 1.87 \\
\midrule
\multicolumn{2}{@{}l}{gain of IGA--KAN over IGA}  & $18.3\times$ & $84.3\times$ & \\
\multicolumn{2}{@{}l}{gain of IGA--KAN over KINN} & $1140\times$ & $568\times$ & \\
\bottomrule
\end{tabular}
\end{table}

\subsection{Time-dependent problems}

\subsubsection{Fisher--Kolmogorov equation}\label{sec:fk}
We begin with the Fisher--Kolmogorov equation, the model of prion-like protein spreading in neurodegeneration, in
the verification test of Corti, Bonizzoni, Ded\`e, Quarteroni and Antonietti \cite{corti2023fk} (their Test case 1),
first on its own domain, the unit square, and then on a slice of a human brain segmented from magnetic resonance
images (MRI),
\begin{equation}\label{eq:fk}
  c_t=\kappa\,\Delta c+\alpha\,c\,(1-c)+f\ \text{ in }(0,1)^2\times(0,T],\qquad \kappa=\alpha=1,\qquad T=10^{-3},
\end{equation}
with Dirichlet data, initial data and the source $f$ derived from the exact solution
\begin{equation}\label{eq:fk-exact}
  c^\ast=\bigl(\cos\pi x\cos\pi y+2\bigr)\,e^{-t}.
\end{equation}
Corti et al.\ solve it with a polytopal discontinuous Galerkin method and report optimal convergence in the energy
norm at $T$; they report the errors only as convergence plots in their energy norm, without tabulated values, so the
comparison here is with IGA and with the networks. IGA advances
\eqref{eq:fk} by Crank--Nicolson with their step $\Delta t=10^{-5}$ and Newton's method at every step, from the
Ritz projection of the initial datum, to the discrete solution $\ch$ at $T$, which is then super-resolved. Over this horizon the exact solution
decays by $0.1\%$ only, so the test, like that of Corti et al., verifies the spatial accuracy at $T$; it does not
probe the accumulation of error over long times. The hybrid keeps the local fit linear by Newton's rule,
$c(1-c)\approx\ch(1-\ch)+(1-2\ch)(\Kv-\ch)$, and treats the time derivative as a given function: its semi-discrete
value at $T$ is obtained from the Galerkin equations of the solve by one mass-matrix solve, with the boundary values
of $\partial_tc^\ast$, and is moved to the right-hand side (the coupled system of Section~\ref{sec:adr} is treated in
the same way). The linearized reaction coefficient $\alpha(2\ch-1)$ is positive, since $c^\ast\ge e^{-T}>\tfrac12$
and $\ch$ is close to it, so the residual test (c) has its maximum-principle constant \eqref{eq:stab}, with $c_0$ the minimum of this coefficient
over the sample points of the patch. The two networks take $(x,y,t/T)$ as input,
with the residual of \eqref{eq:fk} in physical units, as on the coupled system (where $T=1$).

\paragraph{(a) Unit square} On the reference mesh, $p=2$ on $16\times16$ elements ($324$ coefficients), the hybrid reduces the error of the IGA solution by
a factor of $27$ in $L^2$ and of $148$ in $H^1$ (Table~\ref{tab:fk}a), and on every mesh from $4\times4$ to $32\times32$ at $p=2$ and $3$
the gain is $11$ to $54$ in $L^2$ and $15$ to $265$ in $H^1$, the safeguard accepting all patches but one;
the two networks stay at $L^2$ errors of $10^{-2}$ to $10^{-1}$, three to four orders of magnitude above IGA and four to five above the hybrid; KINN's layer is the better of the two here, its three seeds ending between $\sci{1.1}{-2}$ and $\sci{8.4}{-2}$ in $L^2$. With the budget of $1000+200$ iterations, the horizon $T=10^{-3}$ of Corti et al.\ and the scaled input $t/T$, the time derivative enters the residual with the factor $1/T=10^{3}$, and neither network has converged within the budget: over the last $160$ L-BFGS iterations the loss of the KAN still falls from $47$ to $15$ (seed $0$). The network errors are those of this budget. As in the previous examples, the errors of the networks are smooth structures of the size of the
solution's features, the IGA error is the element-scale pattern of the spline space, and the hybrid removes most of
it (Fig.~\ref{fig:fk}).

\begin{figure}[!htbp]
\centering
\includegraphics[width=\linewidth]{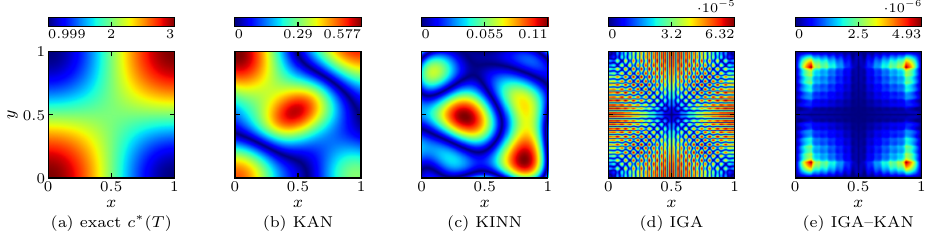}
\caption{Fisher--Kolmogorov equation on the unit square at $t=T$, $p=2$ on $16\times16$ elements: (a) the exact solution; (b)--(e) the
pointwise absolute errors of the KAN, of KINN, of IGA and of IGA--KAN, each on its own color bar (networks: the seed
with the median $L^2$ error). The upper end of each error color bar is the maximum pointwise absolute error of that
panel; the relative errors are given in Table~\ref{tab:fk}a.}
\label{fig:fk}
\end{figure}

\paragraph{(b) Brain slice from MRI} Corti et al.\ also apply the equation to brain geometries reconstructed from MRI
of the OASIS-3 database, a sagittal slice and a three-dimensional brain, in simulations of protein spreading without
an exact solution and on data that, according to the article's data-availability statement, its authors are not
permitted to share. To test the hybrid on an anatomy of that kind we pose the same verification problem on a slice
of a real brain. The slice is taken from the public MRI2FEM data set \cite{mardal2022mri2fem,mri2femdata}, which
holds MR image series under two subject labels and their FreeSurfer reconstructions \cite{fischl2012freesurfer}; we
use the reconstruction labeled \texttt{abby}, its automatic segmentation and its $T_1$-weighted volume in
FreeSurfer's conformed space ($1$~mm voxels), and take the axial slice in which the lateral ventricles have their
largest area. The computational domain is the tissue between two closed curves about the centroid of the ventricles:
the radial envelope of the lateral ventricles with their choroid plexus, after a morphological closing ($3\times3$
structuring element, six iterations) that fills the gaps between the two ventricles, and that of the segmented brain
with its holes filled, each low-pass filtered ($6$ and $10$ Fourier harmonics) and fitted with a periodic cubic
B-spline with $32$ control points (Fig.~\ref{fig:fk-brain}a,b). It is one spline patch, periodic in the angular
direction, and the local models next to the seam are fitted on supports that wrap around it. The envelopes follow the
outline of the brain and of the ventricles but bridge the cortical folds: every point of the two spline curves lies
within $4.6$~mm of the voxel contour of its mask; in the other direction the voxel contour lies up to $9.9$~mm from
the inner curve, which cuts off the anterior tips of the ventricles, and up to $36$~mm from the outer one, the depth
of the bridged interhemispheric fissure; the symmetric mean distance, the average of the two one-sided means, is
$3.1$~mm for the outer curve and $1.9$~mm for the inner one; and $91\%$ of the domain is segmented brain tissue.
The domain, $152$~mm in its larger, antero-posterior extent, is scaled uniformly into the unit square, and
\eqref{eq:fk}--\eqref{eq:fk-exact} are solved on it unchanged, with Dirichlet data on both curves. This is the
verification problem of Corti et al.\ on a real anatomy; their simulations of protein spreading, which have no exact
solution, were not attempted.

On the reference mesh, $p=2$ on $32\times8$ elements ($320$ coefficients), the hybrid reduces the error of the IGA solution by a
factor of $11$ in $L^2$ and of $13.5$ in $H^1$, less than on the square (Table~\ref{tab:fk}b, Fig.~\ref{fig:fk-brain}), and on every mesh from $16\times4$ to
$128\times32$ at $p=2$ and $3$ the gain is $7.7$ to $25$ in $L^2$ and $13.5$ to $23$ in $H^1$, the safeguard accepting all but
$13$ of the $11\,488$ vertex patches (the vertices on the seam are counted at both of its ends). On the slice only the KAN was trained (KINN's layer was not run there), with the loss, optimizer and budget of part (a) and the same numbers of points, $27\,000$ in the interior, placed at Gauss points of the parametric ring, and $2420$ on the two curves and the initial face. Its three seeds end at $L^2$ errors of $\sci{1.2}{-1}$ to $\sci{1.6}{-1}$, two orders of magnitude above IGA and three above the hybrid, and, as on the square, its training has not converged within the budget (the loss still falls from $13$ to $10$ over the last $160$ L-BFGS iterations of seed $0$).

\begin{figure}[!htbp]
\centering
\includegraphics[width=\linewidth]{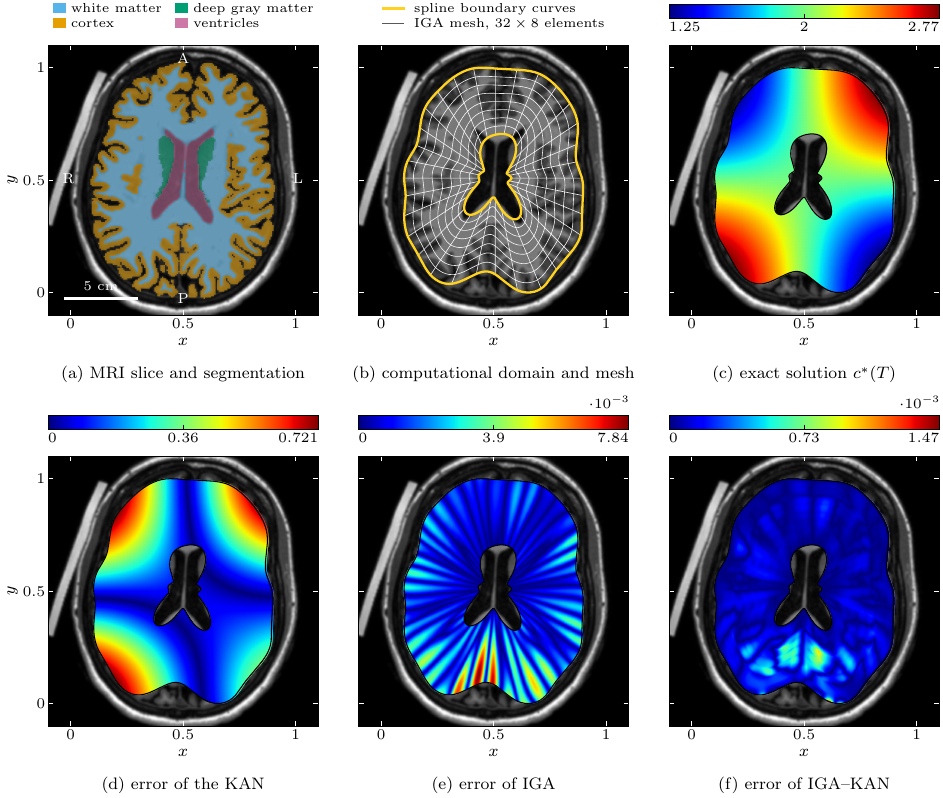}
\caption{Fisher--Kolmogorov equation on a brain slice from MRI, all panels drawn on the $T_1$-weighted slice of the
MRI2FEM data set \cite{mri2femdata} (subject \texttt{abby}, CC BY 4.0) in the coordinates of the unit square. (a) The slice with its FreeSurfer segmentation. (b) The two
spline curves that bound the computational domain and the reference mesh, $p=2$ on $32\times8$ elements. (c) The
exact solution at $t=T$. (d)--(f) The pointwise absolute errors at $t=T$ of the KAN (the seed with the median $L^2$
error), of IGA and of IGA--KAN, each on its own color bar, whose upper end is the maximum pointwise absolute error
of that panel; the relative errors are given in Table~\ref{tab:fk}b.}
\label{fig:fk-brain}
\end{figure}

\begin{table}[!htbp]
\centering
\caption{Fisher--Kolmogorov equation: comparison of the methods at $t=T$ on the reference meshes, (a) on the unit
square ($p=2$, $16\times16$) and (b) on the brain slice from MRI ($p=2$, $32\times8$). The times of the networks
are wall-clock times of single training runs, the median of three seeds, and vary between seeds (KINN: $3069$ to
$4970$~s).}
\label{tab:fk}
\small
\begin{tabular}{@{}lrccr@{}}
\toprule
method & size & $\eLtwo$ & $\eHone$ & time [s] \\
\midrule
\multicolumn{5}{@{}l}{\emph{(a) unit square}}\\
KAN $[3,10,10,1]$  & 1400 parameters & $\sci{1.35}{-1}$ & $\sci{4.41}{-1}$ & 1883 \\
KINN               & 1540 parameters & $\sci{2.24}{-2}$ & $\sci{1.51}{-1}$ & 3820 \\
IGA                & 324 coefficients & $\sci{1.51}{-5}$ & $\sci{1.06}{-3}$ & 0.49 \\
IGA--KAN           & 324 coefficients & $\mathbf{\sci{5.57}{-7}}$ & $\mathbf{\sci{7.14}{-6}}$ & 2.07 \\
\cmidrule(lr){1-5}
\multicolumn{2}{@{}l}{gain of IGA--KAN over IGA}  & $27.1\times$ & $148\times$ & \\
\multicolumn{2}{@{}l}{gain of IGA--KAN over KINN} & $40\,200\times$ & $21\,100\times$ & \\
\midrule
\multicolumn{5}{@{}l}{\emph{(b) brain slice from MRI}}\\
KAN $[3,10,10,1]$  & 1400 parameters & $\sci{1.45}{-1}$ & $\sci{6.41}{-1}$ & 1582 \\
IGA                & 320 coefficients & $\sci{8.92}{-4}$ & $\sci{3.42}{-2}$ & 0.54 \\
IGA--KAN           & 320 coefficients & $\mathbf{\sci{7.96}{-5}}$ & $\mathbf{\sci{2.53}{-3}}$ & 2.37 \\
\cmidrule(lr){1-5}
\multicolumn{2}{@{}l}{gain of IGA--KAN over IGA}  & $11.2\times$ & $13.5\times$ & \\
\multicolumn{2}{@{}l}{gain of IGA--KAN over KAN} & $1\,830\times$ & $254\times$ & \\
\bottomrule
\end{tabular}
\end{table}

\subsubsection{High-frequency heat equation}\label{sec:heat}
KINN demonstrates the spectral bias of multilayer perceptrons on a heat equation of frequency $F=50$, that is, $25$ periods in space
\cite{wang2025kinn,wang2021eigenvector},
\begin{equation}\label{eq:heat}
  u_t=\frac{u_{xx}}{(F\pi)^2}\ \text{ in }(0,1)^2,\qquad u(x,0)=\sin(F\pi x),\quad u(0,t)=u(1,t)=0,\qquad
  u^\ast=e^{-t}\sin(F\pi x).
\end{equation}
It is solved in space--time with the time-upwind Galerkin scheme of Langer, Moore and Neum\"uller \cite{langer2016}.
The three networks follow KINN's protocol: $3000$ Adam epochs with learning rate $10^{-3}$ on $4000$ interior points
resampled at every epoch, three seeds each. The perceptron and the KAN built from our layer both return $u\approx0$;
the KAN built from KINN's layer resolves the $25$ periods, IGA is $54$ times more accurate than it in $L^2$, and
IGA--KAN $829$ times (Fig.~\ref{fig:heat}, Table~\ref{tab:heat}). On the reference mesh the hybrid reduces the error
of IGA by a factor of $15.5$ in $L^2$ and of $32.4$ in $H^1$. On $128\times8$ and $512\times8$ elements the gains at
$p=3$ are $1.28$ and $8.6$ in $L^2$ and $1.28$ and $24$ in $H^1$, on $256\times16$ elements they are $5.8$ and $6.9$,
and at $p=4$ they are $1.3$ to $1.6$ in both norms on the three meshes with eight elements in time.

Figure~\ref{fig:heat} serves two purposes. Panels (a)--(e) show the spectral bias that KINN demonstrates with this
benchmark: the error of the perceptron is the solution itself, the KAN built from KINN's layer leaves pointwise errors up to
$2\times10^{-2}$, IGA up to $2\times10^{-4}$ and IGA--KAN up to $2.5\times10^{-5}$. Panels (f) and (g) show where the
hybrid gains: on meshes with up to about three elements per wavelength, in the upper left, the two maps coincide and
the gain is $1.00$; the gain appears only once the wave is resolved. In the opposite corner both maps level off at
about $\sci{2}{-7}$: at $F=5$ the error of IGA falls only from $\sci{2.9}{-7}$ to $\sci{2.1}{-7}$ between $128$ and
$256$ elements in $x$, and on $256$ elements the two maps coincide again. The figure is drawn like KINN's Fig.~3 (viridis,
a linear color bar per panel, $t$ horizontal), except that (f) and (g) share one logarithmic scale, because their
values span seven orders of magnitude.

\begin{figure}[!htbp]
\centering
\includegraphics[width=\linewidth]{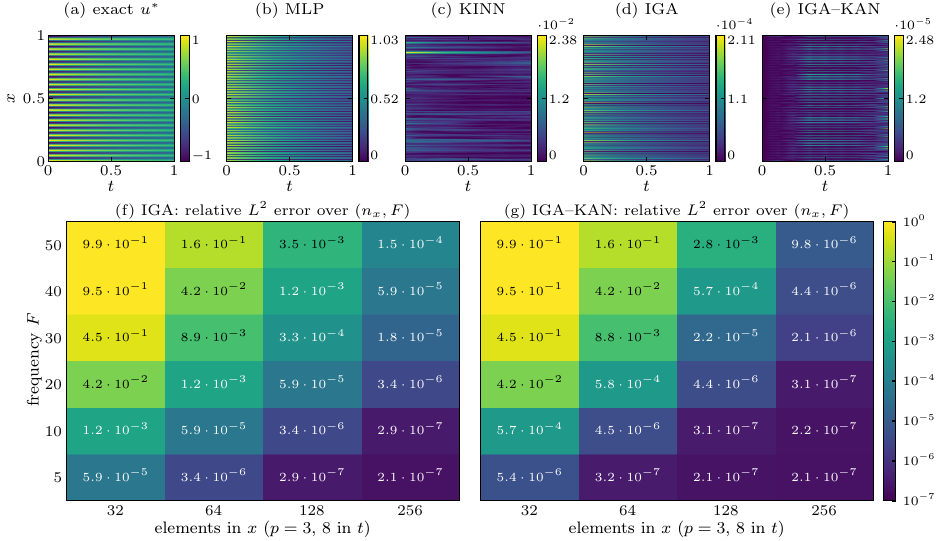}
\caption{High-frequency heat equation, $F=50$: (a) the exact solution; (b)--(e) the pointwise absolute errors of
KINN's perceptron, of the KAN built from KINN's layer, of IGA and of IGA--KAN at $p=3$ on $256\times8$ elements, each
on its own linear color bar, whose upper end is the maximum pointwise absolute error $\max|u-u^\ast|$ of that
panel; (f), (g) the relative $L^2$ error $\eLtwo$ of IGA and of IGA--KAN over the number of elements in $x$ and
the frequency $F$ ($p=3$, $8$ elements in $t$), on one shared logarithmic scale. The numbers printed in (f) and (g)
are relative $L^2$ errors, the measure of Table~\ref{tab:heat}; the color-bar maxima of (b)--(e) are not.}
\label{fig:heat}
\end{figure}

\begin{table}[!htbp]
\centering
\caption{High-frequency heat equation, $F=50$: comparison of the methods on the reference mesh ($p=3$, $256\times8$
space--time elements).}
\label{tab:heat}
\small
\begin{tabular}{@{}lrccr@{}}
\toprule
method & size & $\eLtwo$ & $\eHone$ & time [s] \\
\midrule
MLP $[2,30,30,30,30,1]$ (KINN's)    & 2911 parameters & $1.00$ & $1.00$ & 844 \\
KAN $[2,5,5,1]$, grid 100, our layer & 4160 parameters & $1.00$ & $1.00$ & 12\,177 \\
KINN: the same KAN, KINN's layer    & 4200 parameters & $\sci{8.09}{-3}$ & $\sci{8.46}{-3}$ & 11\,440 \\
IGA      & 2849 coefficients & $\sci{1.51}{-4}$ & $\sci{1.46}{-3}$ & 0.11 \\
IGA--KAN & 2849 coefficients & $\mathbf{\sci{9.76}{-6}}$ & $\mathbf{\sci{4.51}{-5}}$ & 36.0 \\
\midrule
\multicolumn{2}{@{}l}{gain of IGA--KAN over IGA}  & $15.5\times$ & $32.4\times$ & \\
\multicolumn{2}{@{}l}{gain of IGA--KAN over KINN} & $829\times$  & $188\times$ & \\
\bottomrule
\end{tabular}
\end{table}

\subsubsection{Coupled advection--diffusion--reaction system}\label{sec:adr}
Finally we consider the stiff coupled system of Asmouh and Ostermann \cite{asmouh2024nurbs},
\begin{equation}\label{eq:adr}
\begin{aligned}
  &u_t+\mathbf V\cdot\nabla u=d_1\Delta u-bu+v,\qquad v_t+\mathbf V\cdot\nabla v=d_2\Delta v-cv
  \quad\text{in }(0,2\pi)^2\times(0,1],\\
  &d_1=d_2=\tfrac12,\qquad \mathbf V=(\tfrac12,\tfrac12),\qquad b=100,\qquad c=1,
\end{aligned}
\end{equation}
with Dirichlet and initial data from the exact solution
\begin{equation}\label{eq:adr-exact}
  u^\ast=\bigl(e^{-101t}+e^{-2t}\bigr)\cos(x+y-t),\qquad v^\ast=99\,e^{-2t}\cos(x+y-t).
\end{equation}
The two fields are advanced by Crank--Nicolson with Richardson extrapolation and super-resolved at $t=1$, $v$ first
and then $u$ with the super-resolved $v$ in its source. On the reference mesh the hybrid reduces the error of both
fields, by a factor of $5$ in $L^2$ and of $12$ to $27$ in $H^1$, while both networks remain at errors of order $0.4$
(Fig.~\ref{fig:adr}, Table~\ref{tab:adr}). The figure shows the difference in kind: the errors of the networks are
large smooth blobs of the size of the solution's features, the IGA error is the element-scale pattern along the
traveling wave, and the hybrid lowers its maximum by a factor of about eight for $u$ and seven for $v$. On the
coarsest mesh of the study, $8\times8$ elements, the safeguard accepts no patch and the
hybrid coincides with IGA (Section~\ref{sec:mechanism:regimes}); on $16\times16$ and $32\times32$ elements at $p=2$, $3$ and $4$ the gains are $4.9$ to $20$
in $L^2$ and $9.2$ to $104$ in $H^1$.

\begin{figure}[!htbp]
\centering
\includegraphics[width=\linewidth]{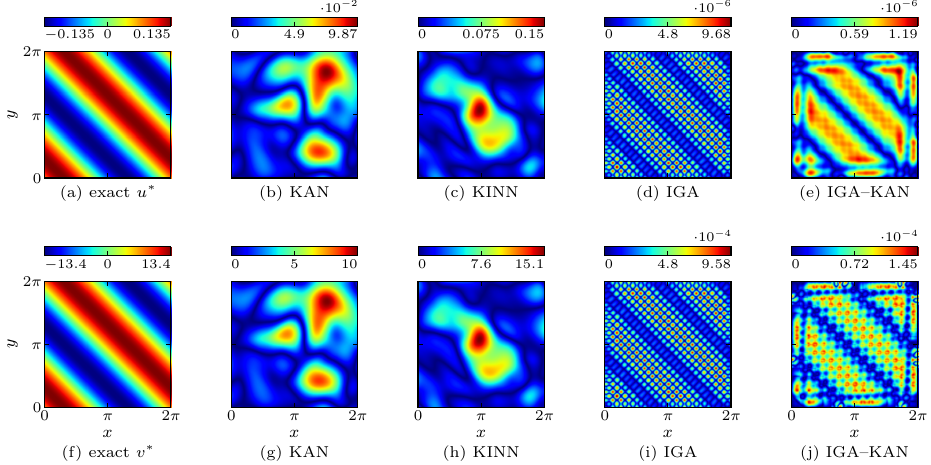}
\caption{Coupled advection--diffusion--reaction system at $t=1$, $p=3$ on $16\times16$ elements; top row $u$, bottom
row $v$: the exact fields and the pointwise absolute errors of the KAN, of KINN, of IGA and of IGA--KAN, each on its
own color bar. The upper end of each error color bar is the maximum pointwise absolute error, $\max|u-u^\ast|$ or
$\max|v-v^\ast|$, of that panel; the relative errors $\eLtwo$ and $\eHone$ of both fields are given in
Table~\ref{tab:adr}.}
\label{fig:adr}
\end{figure}

\begin{table}[!htbp]
\centering
\caption{Coupled advection--diffusion--reaction system: comparison of the methods at $t=1$ on the reference mesh
($p=3$, $16\times16$).}
\label{tab:adr}
\small
\begin{tabular}{@{}lrccccr@{}}
\toprule
 & & \multicolumn{2}{c}{field $u$} & \multicolumn{2}{c}{field $v$} & \\
\cmidrule(lr){3-4}\cmidrule(lr){5-6}
method & size & $\eLtwo$ & $\eHone$ & $\eLtwo$ & $\eHone$ & time [s] \\
\midrule
KAN $[3,10,10,2]$ & 1500 parameters & $\sci{3.66}{-1}$ & $\sci{4.14}{-1}$ & $\sci{3.74}{-1}$ & $\sci{4.24}{-1}$ & 4068 \\
KINN     & 1650 parameters & $\sci{4.28}{-1}$ & $\sci{4.96}{-1}$ & $\sci{4.41}{-1}$ & $\sci{5.06}{-1}$ & 10\,233 \\
IGA      & 361 coefficients & $\sci{3.19}{-5}$ & $\sci{2.90}{-4}$ & $\sci{3.19}{-5}$ & $\sci{2.90}{-4}$ & 0.99 \\
IGA--KAN & 361 coefficients & $\mathbf{\sci{5.92}{-6}}$ & $\mathbf{\sci{1.06}{-5}}$ & $\mathbf{\sci{6.48}{-6}}$ & $\mathbf{\sci{2.36}{-5}}$ & 9.24 \\
\midrule
\multicolumn{2}{@{}l}{gain of IGA--KAN over IGA}  & $5.39\times$ & $27.3\times$ & $4.92\times$ & $12.3\times$ & \\
\multicolumn{2}{@{}l}{gain of IGA--KAN over KINN} & $72\,300\times$ & $46\,700\times$ & $68\,100\times$ & $21\,400\times$ & \\
\bottomrule
\end{tabular}
\end{table}

%% file: sections/conclusions.tex
\section{Conclusions and future work}\label{sec:conclusions}

We kept the classical solver and used a Kolmogorov--Arnold network only where its prediction passes a computable check. A
standard isogeometric Galerkin solve on the exact NURBS geometry provides $\uh$; on every knot-vertex patch one
small Kolmogorov--Arnold ridge model is fitted to the strong form of the equation, to the exact boundary data and
to $\uh$, in one batched linear least-squares problem, and an \emph{a posteriori} safeguard decides patch by patch
whether the local model may replace the isogeometric solution. The post-process therefore has no optimizer, no learning rate and no
initialization, and wherever no local model is accepted the isogeometric solution is kept.

We tested the method on eight benchmarks: three Poisson problems on curved domains, one of them three-dimensional and one also
inverse, the Laplace equation on a complex geometry, the nonlinear Bratu equation, the Fisher--Kolmogorov equation on
the unit square and on a domain fitted to a brain slice from magnetic resonance images, a high-frequency heat equation and a stiff
coupled advection--diffusion--reaction system. On these benchmarks the hybrid reduced the error of the IGA solution it post-processes
by factors of $4.2$ to $90$ in $\Ltwo$ and $4.1$ to $220$ in $\Hone$ on the reference meshes (the upper ends on the
disk, whose solution is a quadratic polynomial; $4.2$ to $27$ and $4.1$ to $148$ on the six benchmarks with
non-polynomial solutions, all but the disk and the thick ring), at an unchanged number of unknowns of the Galerkin solve and at $4$ to about $500$ times
its time, and its $\Ltwo$ error was $6$ to $6\times10^{4}$ times smaller than that of the best Kolmogorov--Arnold
network we trained from scratch on the same equations with a fixed training protocol; these factors describe that
protocol, not the best accuracy such networks can reach. On the disk of PI-GGN and on the flower of KINN and PG-KINN it is
more accurate than the published results, and on the disk it identifies an unknown source from a single observation
by one division, with an error $167$ times smaller than IGA.

The analysis of Section~\ref{sec:mechanism} attributes these gains to the superconvergence of local averages of the
Galerkin solution, which the local fit recovers with a kernel chosen by the equation, and bounds the error of the
blended solution by the residual of the local models, which the safeguard computes, and by their error on the patch
boundaries, for which the consensus of overlapping patches is a computable surrogate. With the settings used throughout, the safeguarded
hybrid did not return a larger error than the isogeometric solution, beyond round-off, on any mesh of the
benchmarks.

\paragraph{Future work} The results suggest the following directions.
\begin{enumerate}
\item \emph{Efficiency.} The fits of different patches are independent, so the post-process maps directly onto a
  GPU, and on the first benchmark a non-redundant basis of the local polynomial space reproduced the result at a
  quarter of the cost; how far the two together reduce the cost of the post-process remains to be measured.
\item \emph{Coarse meshes.} On coarse meshes (up to about three elements per wavelength on the heat equation, and
  on the coarsest mesh of the coupled system) the safeguard returns the isogeometric solution. The patches that
  are not accepted mark where a smooth local model fails the tests, and this map needs no exact solution, which suggests the
  safeguard as a solve-free indicator
  for adaptive refinement; coupling the two, and choosing the patch size, the local degree and the number of sample
  points automatically (in three dimensions the gain depends on them, Section~\ref{sec:thickring}), is a way to
  obtain gains in this regime.
\item \emph{Non-smooth solutions.} Corner and crack singularities and material interfaces are outside a polynomial
  local space. Enriching the local models with the singular functions, where the exponent is known, or combining
  the post-process with graded meshes, is the natural extension. Outer splines with interior knots, which leave the
  polynomial space while the fit stays linear, are the part of the Kolmogorov--Arnold form not used in the method (on
  the quarter annulus they change the errors by less than $5\%$, Remark~\ref{rem:space}); whether they help near
  layers and interfaces remains to be tested.
\item \emph{Vector-valued and higher-order problems.} The ridge identities extend to linear elasticity and to
  fourth-order operators; what is missing there are counterparts of the residual and $p$-consistency tests, without
  which the safeguard does not carry over as it stands.
\item \emph{Sharper guarantees.} Replacing the consensus test by a computable bound on the boundary term of
  \eqref{eq:hybrid-bound} would turn the safeguard from a filter into a guaranteed error estimate, and a quantitative
  account of the gain in terms of the negative-norm estimates would predict it \emph{a priori}; both could also relax
  the $p$-consistency test, which is conservative where the right-hand side vanishes (Section~\ref{sec:flower}).
\item \emph{Noisy data.} In the inverse problem of Section~\ref{sec:inverse} the advantage of the hybrid decreases
  once the noise in the observation approaches the discretization error; a noise-aware weighting of the data rows is
  the natural remedy.
\end{enumerate}